\documentclass[10pt,a4paper,reqno,twoside]{amsart}

\newcommand{\usepackageifexists}[1]{%
    \IfFileExists{#1.sty}{\usepackage{#1}}%
       {\GenericInfo{taglia}{Il package #1 non esiste.}}}

\usepackage{verbatim}
\usepackage{amssymb}
\usepackage{upref}
\usepackage{amsmath}
\usepackage{amsthm}
\usepackage{amsfonts}
\usepackage{latexsym}
\usepackage{amssymb}
\usepackage{color}
\usepackage{amsmath}
\usepackage{amsfonts}
\usepackage{amsthm}
\usepackage{amssymb}
\usepackage{amsmath,amsfonts,amscd,bezier}
\usepackage[latin1]{inputenc}
\usepackage{enumerate}

\theoremstyle{plain}
\newtheorem{Thm}{Theorem}
\newtheorem*{Thm*}{Theorem}

\newtheorem{citThm}{Theorem}%[chapter]%[section]
\newtheorem{Prop}{Proposition}%[section]
\newtheorem{Lem}[Prop]{Lemma}

\theoremstyle{definition}
\newtheorem{Def}{Definition}
\newtheorem{Hyp}{Hypothesis}

\newtheorem{Not}{Notation}
\theoremstyle{remark}
\newtheorem{Rem}{Remark}

\newtheorem{Example}{Example}

\allowdisplaybreaks

\newcommand{\N}{\mathbb{N}}

\newcommand{\R}{\mathbb{R}}
\newcommand{\Rn}{{\R^n}}

\newcommand{\ee}{{\varepsilon}}

\newcommand{\bs}{\eta} %THIS MACRO IS FOR THE SHAPE FUNCTION OF $b(t)$
\newcommand{\homog}{\lambda} % THIS MACRO IS FOR THE QUASI HOMOGENEOUS ASSUMPTION ON $b(t)$
\newcommand{\dat}{g} % THIS MACRO IS FOR THE NOT HOMOG DATA OF THE LINEARIZED PROBLEM

\newcommand{\xii}{{\abs{\xi}}}

\DeclareMathOperator{\diag}{diag}

\DeclareMathOperator{\supp}{supp}

\renewcommand{\doteq}{:=}

\newcommand{\pd}{{\mathrm{pd}}}
\newcommand{\hyp}{{\mathrm{hyp}}}

\newcommand{\elli}{{\mathrm{ell}}}
\newcommand{\red}{{\mathrm{red}}}
\newcommand{\comp}{{\mathrm{comp}}}

\newcommand{\expo}{{\kappa}}

\newcommand{\Fuj}{{\mathrm{Fuj}}}

\newcommand{\GN}{{\mathrm{GN}}}

\newcommand{\lin}{{\mathrm{lin}}}
\newcommand{\nl}{{\mathrm{nl}}}

\def\<#1\>{\left\langle#1\right\rangle }

\newcommand{\abs}[1]{\left\vert#1\right\vert}

\begin{document}

\title[Semi-linear Wave Equations]
          {Semi-linear Wave Equations\\
          with Effective damping}
\author[M. D'Abbicco, S. Lucente, M. Reissig]%
    {Marcello D'Abbicco, Sandra Lucente, Michael Reissig}

\address{Marcello D'Abbicco and Sandra Lucente, Department of Mathematics, University of Bari, Via E. Orabona 4 - 70125 BARI -
ITALY}

\address{Michael Reissig, Faculty for Mathematics and Computer
Science, Technical University Bergakademie Freiberg, Pr\"uferstr.9 -
09596 FREIBERG - GERMANY}

\begin{abstract}
We study the Cauchy problem for the semi-linear damped wave equation
\[ u_{tt}-\triangle u+b(t)u_t=f(u), \qquad u(0,x)=u_0(x), \quad u_t(0,x)=u_1(x), \]
in any space dimension~$n\geq1$. We assume that the time-dependent
damping term~$b(t)>0$ is \emph{effective}, in particular
$tb(t)\to\infty$ as $t\to \infty$. We prove the global existence of
small energy data solutions for~$|f(u)|\approx |u|^p$ in the
supercritical case $p>1+2/n$ and $p\leq n/(n-2)$ for $n\ge 3$.\\
\end{abstract}



\maketitle

%\makeendnotes

We consider the Cauchy problem for the dissipative semi-linear
equation
\begin{equation}
\label{eq:diss}
\begin{cases}
u_{tt}-\triangle u+b(t)u_t=f(u), & t\geq0, \ x\in\R^n,\\
u(0,x)=u_0(x), \\
u_t(0,x)=u_1(x),
\end{cases}
\end{equation}
where the time-dependent damping term~$b(t)>0$ is \emph{effective}, in particular $tb(t)\to\infty$ as $t\to \infty$, and the nonlinear term satisfies
\begin{equation}\label{eq:disscontr}
f(0)=0, \qquad |f(u)-f(v)|\lesssim |u-v|(|u|+|v|)^{p-1},
\end{equation}
for a given $p>1$. Our aim is to establish the existence
of~$\mathcal{C}([0,\infty),H^1)\cap\mathcal{C}^1([0,\infty),L^2)$
solutions of~\eqref{eq:diss} assuming small initial data in the
energy space~$H^1\times L^2$ or in some weighted energy spaces.
Clearly this will require suitable assumptions on~$b(t)$ and on the
exponent~$p$ in~\eqref{eq:disscontr}. In Section~\ref{sec:classic}
we first present some results related to the the semi-linear wave
equation with a constant damping term. We refer the interested
reader to~\cite{ITY, Nx} and to the quoted references for the damped
wave equation with $x$-dependent damping term~$b(x) u_t$. In Section
\ref{sec:Main} we state our main theorems and some auxiliary
results.

\section{The classical semi-linear damped wave equation}\label{sec:classic}

Many papers concern with the \emph{classical} semi-linear damped
wave equation, i.e. with the case~$b\equiv 1$:
\begin{equation}\label{eq:classic}
\begin{cases}
u_{tt}-\triangle u+u_t=f(u),\\
u(0,x)=u_0(x), \\
u_t(0,x)=u_1(x).
\end{cases}
\end{equation}
For the sake of clarity we put
\begin{align*}
p_\GN(n)
    & =1+\frac2{n-2}=\frac{n}{n-2}, & \text{for~$n\geq3$,}\\
p_\Fuj(n)
    & =1+\frac{2}{n}, & \text{for~$n\geq1$.}
\end{align*}
As stated in \cite{NO}, for initial data~$(u_0,u_1)\in H^1\times
L^2$ with compact support in~$B_K(0)$, and $p\le p_\GN(n)$
if~$n\geq3$, the problem~\eqref{eq:classic} admits a unique local
solution
$u\in\mathcal{C}([0,T_m),H^1)\cap\mathcal{C}^1([0,T_m),L^2)$ for
some maximal existence time $T_m\in(0,+\infty]$  and for any
~$t<T_m$ it holds~$\supp u(t,\cdot)\subset B_{K+t}(0)$.
\\
One of the first results on global existence theory has been given
in~\cite{NO} establishing global existence for small data by using
the technique of \emph{potential well} and \emph{modified potential
well}. Let~$\widetilde{W}\subset H^1$ be the interior
of the set
\[ \left\{ u\in H^1 : \ \|\nabla u\|_{L^2}^2 \geq \|u\|_{L^{p+1}}^{p+1} \right\}. \]
In particular, by assuming $(u_0,u_1)\in \widetilde{W}\times L^2$ the authors
remove the compactness assumption on the support of the data and
they prove the local existence of the solution, provided that $p<
(n+2)/(n-2)$ if~$n\geq3$ (Theorem 1 in~\cite{NO}). In Theorem~$3$ of
the same paper, they prove the global existence, provided that the
data in~$\widetilde{W}\times L^2$ satisfies energy smallness
assumptions and the exponent satisfies~$p\geq 1+4/n$ with
$p<(n+2)/(n-2)$ if~$n\geq3$ (we remark that this set is not empty).
In such a case, the energy of the solution to~\eqref{eq:diss}
satisfies the same decay estimates of the linear equation, i.e.
$\|u_t(t,\cdot)\|_{L^2}^2+\|\nabla u(t,\cdot)\|_{L^2}^2\leq
C(1+t)^{-1}$.
\\
Assuming compactly supported data~$(u_0,u_1)\in H^1\times L^2$
pointwise sufficiently small, a global existence result for
$p>p_\Fuj(n)$, and $p\leq p_\GN(n)$ if~$n\geq3$, has been proved
in~\cite{TY} (we remark that this set is never empty). The approach
followed in~\cite{TY} makes use of the Matsumura
estimates~\cite{Matsu} for the solution to the Cauchy problem for
the classical damped linear wave equation
\begin{equation}\label{eq:lindamped}
u_{tt}-\triangle u+u_t=0, \qquad u(0,x)=u_0(x), \quad
u_t(0,x)=u_1(x).
\end{equation}
In order to state these estimates we define
\begin{gather}
\label{eq:Amk}
\mathcal{A}_{m,k} \doteq (L^m\cap H^k) \times (L^m\cap H^{k-1}),\\
\label{eq:normAmk} \|(u,v)\|_{\mathcal{A}_{m,k}} \doteq
\|u\|_{L^m}+\|u\|_{H^k}+\|v\|_{L^m}+\|v\|_{H^{k-1}}
\end{gather}
for~$m\in [1,2)$ and~$k\in\N$. If~$(u_0,u_1)\in \mathcal{A}_{m,1}$ for some~$m\in [1,2)$, then the solution to~\eqref{eq:lindamped} satisfies
\begin{equation}
\label{eq:matsu}
\begin{array}{rl}
\|u(t,\cdot)\|_{L^2} &\leq C (1+t)^{-\frac{n}2(\frac1m-\frac12)} \|(u_0,u_1)\|_{\mathcal{A}_{m,0}},\\
\| \nabla u(t,\cdot) \|_{L^2} &\leq C (1+t)^{-\frac{n}2(\frac1m-\frac12)-\frac12} \|(u_0,u_1)\|_{\mathcal{A}_{m,1}}, \\
\|u_t(t,\cdot)\|_{L^2} &\leq C (1+t)^{-\frac{n}2(\frac1m-\frac12)-1} \|(u_0,u_1)\|_{\mathcal{A}_{m,1}}.
\end{array}
\end{equation}
Since in~\cite{TY} the data $(u_0,u_1)\in H^1\times L^2$ has compact
support, the authors apply Matsumura's estimates for~$m=1$.
Moreover, they find that the energy of the solution
to~\eqref{eq:classic} satisfies~\eqref{eq:matsu} for~$m=1$ and they prove a blow-up
result in finite time if~$p<p_\Fuj(n)$, provided that~$f(u)=|u|^p$ and that
$\int_{\Rn} u_j(x)\,dx>0$ for $j=0,1$. The same result is obtained
in \cite{Z} for the case $p=p_\Fuj(n)$.
\\
We remark that the exponent $p_\Fuj(n)$ is the Fujita's one, the
same which guarantees the existence of a non-negative classical
global solution to the semi-linear heat equation
\[ u_t-\triangle u=u^p, \qquad u(0,x)=u_0(x), \]
provided that~$u_0\ge 0$ is sufficiently smooth. The Fujita exponent is sharp, that is, if~$p\leq p_\Fuj(n)$, the semi-linear heat equation does not admit any global regular solution (see~\cite{Fuj}).

Coming back to the global existence theory for the semi-linear
classical damped wave equation, the condition on the compact support
of the data has been relaxed in~\cite{IT} by assuming small data in
a suitable weighted Sobolev space:
\begin{equation}
\label{eq:ikeE0} I^2 \doteq \int_{\Rn}  e^{|x|^2/2} \left( |u_1|^2 +
|\nabla u_0|^2 + |u_0|^2 \right) dx \leq \epsilon^2.
\end{equation}
Condition~\eqref{eq:ikeE0} implies that $(u_0,u_1)\in (W^{1,1}\cap H^1)\times(L^1\cap L^2)\subset \mathcal{A}_{1,1}$, therefore in~\cite{IT} the authors can use Matsumura's estimates~\eqref{eq:matsu} for~$m=1$.
\\
Furthermore, in~\cite{IMN} the authors show that the smallness in
weighted Sobolev spaces or compactly supported data can be avoided
assuming smallness in~$\mathcal{A}_{1,1}$ and the critical exponent
remains $p_\Fuj(n)$ for $n=1,2$. Since their technique
requires~$p>2$, the authors obtain global existence \emph{only} for
$2<p\leq 3=p_\GN(3)$ if~$n= 3$ (we remark that~$p_\Fuj(3)=1+2/3<2$).
In~\cite{IO} this result is extended to initial data in $\mathcal
A_{m,1}$ for~$m\in(1,2)$.

In this paper, we are going to follow the approach in~\cite{IMN, IO,
IT}. In particular we are going to use some Matsumura-type estimates
for the linear wave equation with time-dependent \emph{effective}
damping, derived by J. Wirth~\cite{W07}. In order to do this, we are
going to extend these estimates to a family of Cauchy problems with
initial time as parameter.
\\
We remark that the Cauchy problem for the classical wave equation
(i.e.~$b\equiv1$) is independent of translation in time, since the
coefficients of the equation do not depend on~$t$ and hence
Duhamel's principle easily applies, whereas for a non-constant
~$b=b(t)$ the situation is more complicated.

%%%%%%%%%%%%%%%%%%%%%%%%%%%%%%%%%%%%%%%%
%
%%%%%%%%%%%%%%%%%%%%%%%%%%%%%%%%%%%%%%%%

\section{Main Results}\label{sec:Main}

In order to present our results we fix the class of \emph{effective}
damping terms~$b(t)$ which are of interest in the further
discussions.

\begin{Hyp}\label{Hyp:b}
We make the following assumptions on the damping term~$b(t)$:
\begin{enumerate}[(i)]
\item \label{en:bpos} $b(t)>0$ for any~$t\geq0$,
\item \label{en:beff} $b(t)$ is monotone, and $tb(t)\to\infty$ as~$t\to\infty$,
\item \label{en:bsnotdiss} $((1+t)^2b(t))^{-1}\in L^1([0,\infty))$,
\item \label{en:breg} $b\in\mathcal{C}^3$ and
\begin{equation}\label{eq:oscb}
\frac{\abs{b^{(k)}(t)}}{b(t)}\lesssim \frac1{(1+t)^k}\,,
\end{equation}
for any~$k=1,2,3$,
\item \label{en:1b} $1/b\not\in L^1$.
\end{enumerate}
\end{Hyp}
The damping term~$b(t)$ is \emph{effective} according to~\cite{W05, W07}.
\begin{Def}\label{Def:Bt0}
We denote by~$B(t,0)$ the primitive of~$1/b(t)$ which vanishes at~$t=0$, that is,
\begin{equation}\label{eq:B}
B(t,0) = \int_0^t \frac1{b(\tau)}\, d\tau.
\end{equation}
\end{Def}
Thanks to conditions~\eqref{en:bpos} and~\eqref{en:1b} in Hypothesis~\ref{Hyp:b}, $B(t,0)$ is a positive, strictly increasing function, and~$B(t,0)\to+\infty$ as~$t\to\infty$.
\\
Let us consider the Cauchy problem for the linear damped wave equation:
\begin{equation}\label{eq:CPlin}
\begin{cases}
u_{tt}-\triangle u +b(t)u_t=0,\\
u(0,x)=u_0(x),\\
u_t(0,x)=u_1(x).
\end{cases}
\end{equation}
In 2005, J. Wirth derived Matsumura-type estimates for the solution to~\eqref{eq:CPlin} (see Theorem 5.5 in~\cite{W05} and Theorem 26 in~\cite{W07}).
\begin{citThm}
If Hypothesis~\ref{Hyp:b} is satisfied and~$(u_0,u_1)\in\mathcal{A}_{m,1}$ for some~$m\in [1,2]$, then the solution to the Cauchy problem~\eqref{eq:CPlin}
satisfies the following decay estimates$:$
\begin{align}
\label{eq:MWu}
\|u(t,\cdot)\|_{L^2}
    & \leq C (1+B(t,0))^{-\frac{n}2(\frac1m-\frac12)} \|(u_0,u_1)\|_{\mathcal{A}_{m,0}}, \\
\label{eq:MWux}
\|\nabla u(t,\cdot)\|_{L^2}
    & \leq C (1+B(t,0))^{-\frac{n}2(\frac1m-\frac12)-\frac12} \|(u_0,u_1)\|_{\mathcal{A}_{m,1}}, \\
\label{eq:MWut}
\|u_t(t,\cdot)\|_{L^2}
    & \leq C  (b(t))^{-1}(1+B(t,0))^{-\frac{n}2(\frac1m-\frac12)-1}  \|(u_0,u_1)\|_{\mathcal{A}_{m,1}}.
\end{align}
\end{citThm}
In order to prove our results for semi-linear damped wave equations we need a further assumption on~$b(t)$ in the case of increasing $b(t)$.
\begin{Hyp}\label{Hyp:furtherb}
Let~$b\in\mathcal{C}^1([0,\infty))$, $b(t)>0$. We assume that there exists a constant ~$m\in [0,1)$ such that
\begin{equation}\label{eq:tbprimeb}
tb'(t)\leq mb(t), \qquad t\geq 0.
\end{equation}
\end{Hyp}
\begin{Rem}
We recall that if~$b(t)$ is as in Hypothesis~\ref{Hyp:b}, then it is either increasing or decreasing. If~$b(t)$ is decreasing, then~\eqref{eq:tbprimeb}
holds for~$m=0$. On the other hand, if~$b(t)$ is increasing, condition~\eqref{eq:tbprimeb} is stronger than the upper bound of~\eqref{eq:oscb} for~$k=1$.
\end{Rem}
Our first result is based on a generalization of the ideas in~\cite{IT}.
\begin{Not}\label{Not:weight}
Given $\rho:\R^n\to [0,\infty)$, we say that $f\in L^q(\rho)$ for some~$q\in [1,\infty]$ if $\rho f\in L^q$. Similarly, for any $f\in L^2(\rho)$ such that
$\nabla f\in L^2(\rho)$ we write $f\in H^1(\rho)$.
\\
It is easy to see that $H^1(\rho)\hookrightarrow H^1$ if $\rho>0$ and $1/\rho\in L^\infty$.
\\
Since in this paper we will work with exponential weight functions, for the sake of brevity we will denote~$L^q(e^g)$ as~$L^q_g$ and~$H^1(e^g)$ as~$H^1_g$
for any~$g:\R^n\to\R$.
\end{Not}
We assume that the initial data of~\eqref{eq:diss} is small in~$H^1_{\alpha|x|^2}\times L^2_{\alpha|x|^2}$ for some~$\alpha\in(0,1/4]$. We put
\begin{equation}\label{eq:I0alpha}
I_\alpha^2 \doteq \int_{\R^n} e^{2\alpha |x|^2}\left(|u_0(x)|^2+|\nabla u_0(x)|^2+|u_1(x)|^2\right) dx.
\end{equation}
\begin{Thm}\label{Thm:main}
Let $n\ge 1$ and $p>p_\Fuj(n)$. Moreover, let~$p\leq p_\GN(n)$ if~$n\geq3$. Let~$\alpha\in(0,1/4]$. Then there exists~$\epsilon_0>0$ such that,
if~$I_\alpha\leq\epsilon_0$, where~$I_\alpha$ is introduced in~\eqref{eq:I0alpha}, then there exists a unique solution to~\eqref{eq:diss}
in~$\mathcal{C}([0,\infty),H^1)\cap\mathcal{C}^1([0,\infty),L^2)$.
\\
Moreover, there exists a constant~$C>0$ such that the solution satisfies the decay estimates
\begin{align}
\label{eq:decayu}
\|u(t,\cdot)\|_{L^2}
    & \leq C\,I_\alpha \,(1+B(t,0))^{-\frac{n}4}, \\
\label{eq:decayux}
\|\nabla u(t,\cdot)\|_{L^2}
    & \leq C\,I_\alpha \,(1+B(t,0))^{-\frac{n}4-\frac12}, \\
\label{eq:decayut}
\|u_t(t,\cdot)\|_{L^2}
    & \leq C\,I_\alpha \,(1+B(t,0))^{-\frac{n}4}(1+t)^{-1}.%(b(t))^{-1}(1+B(t,0))^{-\frac{n}4-1}
\end{align}
Finally, the wave energy is uniformly bounded in the family of
weighted spaces $L^2_{\psi(t,\cdot)}$, where
\begin{equation}\label{eq:psi}
\psi(t,x)=\frac{\alpha|x|^2}{(1+B(t,0))},
\end{equation}
namely,
\[
\int_{\Rn} e^{\frac{2\alpha |x|^2}{(1+B(t,0))}} \left( |\nabla u(t,x)|^2+|u_t(t,x)|^2 \right) dx \leq CI_\alpha^2, \qquad t\geq0.
\]
\end{Thm}
We notice that~$\psi(0,x)=\alpha|x|^2$ gives the weight at~$t=0$.
\\
The decay estimates
\eqref{eq:decayu}-\eqref{eq:decayux}-\eqref{eq:decayut} for the
solution of the semi-linear problem~\eqref{eq:diss} correspond to
the decay estimates \eqref{eq:MWu}-\eqref{eq:MWux}-\eqref{eq:MWut},
with~$m=1$, for the solution of the linear problem~\eqref{eq:CPlin}.
In particular, the decay factor $(1+t)^{-1}$ in \eqref{eq:decayut}
is equivalent to $(b(t))^{-1}(1+B(t,0))^{-1}$ in~\eqref{eq:MWut}, as
we shall see in Remark~\ref{Rem:bminus1B}.

\bigskip

Now let us assume~$(u_0,u_1)\in\mathcal{A}_{1,1}$ (see~\eqref{eq:Amk}). We follow the approach in~\cite{IO} to gain a global existence result for this
larger class of data. This goal will restrict our range of admissible~$n$ and~$p$.
\begin{Thm}\label{Thm:low}
Let~$n\leq4$ and let $:$
\begin{equation}\label{eq:admplow}
\begin{cases}
p>p_\Fuj(n) & \text{if~$n=1,2$,} \\
2\le p\le 3=p_\GN(3) & \text{if~$n=3$,}\\
p=2=p_\GN(4) & \text{if~$n=4$.}
\end{cases}
\end{equation}
Let $(u_0,u_1)\in\mathcal{A}_{1,1}$. Then, there exists~$\epsilon_0>0$ such that, if
\[
\|(u_0,u_1)\|_{\mathcal{A}_{1,1}}\leq\epsilon_0,
\]
then there exists a unique solution to~\eqref{eq:diss} in~$\mathcal{C}([0,\infty),H^1)\cap\mathcal{C}^1([0,\infty), L^2)$. Moreover,
there exists a constant~$C>0$ such that the solution satisfies the decay estimates
\begin{align}
\label{eq:decayulow}
\|u(t,\cdot)\|_{L^2}
    & \leq C\,\|(u_0,u_1)\|_{\mathcal{A}_{1,1}} \,(1+B(t,0))^{-\frac{n}4}, \\%C\,\|(u_0,u_1)\|_{\mathcal{A}_{1,0}} \,(1+B(t,0))^{-\frac{n}4}
\label{eq:decayuxlow}
\|\nabla u(t,\cdot)\|_{L^2}
    & \leq C\,\|(u_0,u_1)\|_{\mathcal{A}_{1,1}} \,(1+B(t,0))^{-\frac{n}4-\frac12}, \\
\label{eq:decayutlow}
\|u_t(t,\cdot)\|_{L^2}
    & \leq C\,\|(u_0,u_1)\|_{\mathcal{A}_{1,1}} \,(1+B(t,0))^{-\frac{n}4} (1+t)^{-1}.
\end{align}
\end{Thm}
As in Theorem~\ref{Thm:main} the solutions to the semi-linear Cauchy problem~\eqref{eq:diss} and to the linear one~\eqref{eq:CPlin} have the same decay
rate.
\begin{Rem}
Since we are interested in energy solutions in Theorem~\ref{Thm:low} the restriction $p\geq2$ appears in
a natural way. In both Theorems~\ref{Thm:main} and~\ref{Thm:low} the Fujita exponent $p_\Fuj(n)$ appears as a lower bound of
admissible exponents~$p$. The optimality of this bound follows from the result in Section~\ref{sec:Optim}.
\end{Rem}

%%%%%%%%%%%%%%%%%%%%%%%%%%%%%%

\subsection{Examples}

\begin{Example}\label{Ex:bk}
Let us choose
\begin{equation}\label{eq:bk}
b(t)=\frac\mu{(1+t)^\expo} \quad \text{for some~$\mu>0$ and~$\expo\in(-1,1)$.}
\end{equation}
Being $\expo\in(-1,1)$, Hypothesis~\ref{Hyp:b} holds. Indeed~$tb(t)\approx (1+t)^{1-\kappa}$ and~$(1+t)^2b(t)\approx (1+t)^{2-\expo}$ as~$t\to\infty$, so that
$1/b\not\in L^1$ and $((1+t)^2b(t))^{-1}\in L^1$.
\\
Hypothesis~\ref{Hyp:furtherb} holds since~\eqref{eq:tbprimeb} is satisfied for~$m=\max\{-\expo,0\}$.
\\
We observe that $1+B(t,0)\approx (1+t)^{1+\expo}$. % with $C>0$ depending on $\expo$ and $\mu$.
Therefore we can apply Theorems~\ref{Thm:main} and~\ref{Thm:low} with~$I_\alpha=\epsilon$ and~$\|(u_0,u_1)\|_{\mathcal{A}_{1,1}}=\epsilon$, respectively.
The decay in \eqref{eq:decayu}-\eqref{eq:decayux}-\eqref{eq:decayut} or in \eqref{eq:decayulow}-\eqref{eq:decayuxlow}-\eqref{eq:decayutlow} can be rewritten as
\begin{align*}
\|u(t,\cdot)\|_{L^2}
    & \leq C\,\epsilon \,(1+t)^{-(1+\expo)\frac{n}4}, \\
\|\nabla u(t,\cdot)\|_{L^2}
    & \leq C\,\epsilon \,(1+t)^{-(1+\expo)(\frac{n}4+\frac12)}, \\
\|u_t(t,\cdot)\|_{L^2}
    & \leq C\,\epsilon \,(1+t)^{-(1+\expo)\frac{n}4-1}.
\end{align*}
In particular, for $\expo=0$ we have a constant coefficient in the damping term and we cover the results described in Section \ref{sec:classic}.
\end{Example}
\begin{Example}\label{Ex:log}
Let us multiply the function~$b(t)$ in~\eqref{eq:bk} by a logarithmic positive power. We consider the following coefficient $b(t)$ in the damping term:
\begin{equation}
\label{eq:bklog} b(t)= \frac\mu{(1+t)^\expo} (\log (c+t))^\gamma \quad \text{for some~$\mu>0, \gamma >0,$ and~$\expo\in(-1,1]$,}
\end{equation}
where~$c=c(\expo,\gamma)>1$ is a suitably large positive constant.
\\
It is easy to check that conditions \eqref{en:bpos}-\eqref{en:breg}-\eqref{en:1b} in Hypothesis~\ref{Hyp:b} hold and that~$tb(t)\to+\infty$
as~$t\to\infty$. Moreover, condition~\eqref{en:bsnotdiss} in Hypothesis~\ref{Hyp:b} holds for any~$\gamma>0$ if~$\expo\in(-1,1)$ and for any~$\gamma>1$
if~$\expo=1$, since
\[ ((1+t)^2b(t))^{-1} = \frac1{\mu (1+t)^{2-\expo}(\log (c+t))^\gamma}. \]
For $\kappa=0$ the assumption~\eqref{en:beff} in Hypothesis~\ref{Hyp:b} is satisfied. Let~$\kappa\in
(-1,1]$, $\kappa\neq0$. If we explicitly compute~$b'(t)$, then we
derive
\begin{align*}
b'(t)
    & = -\frac{\mu\expo}{(1+t)^{\expo+1}} (\log (c+t))^\gamma + \frac{\mu\gamma}{(1+t)^\expo(c+t)} (\log (c+t))^{\gamma-1}\\
    & = \frac{\mu}{(1+t)^{\expo+1}} (\log (c+t))^\gamma \left(-\expo + \frac{\gamma(1+t)}{(c+t)\log(c+t)} \right),
\intertext{therefore we get}
b'(t)
    &  \approx \frac1{(1+t)^{\expo+1}} (\log (c+t))^\gamma \approx \frac{b(t)}{1+t}
\end{align*}
provided that~$c=c(\expo,\gamma)>e^{\frac{\gamma}{|\expo|}}$. We
proved that~$b(t)$ is monotone and this concludes the proof of
Hypothesis~\ref{Hyp:b}.
\\
If~$\expo\in (0,1]$, then Hypothesis~\ref{Hyp:furtherb} holds
since~$b(t)$ is decreasing. If~$\expo\in (-1,0]$,
then~\eqref{eq:tbprimeb} is satisfied
for~$c>e^{\frac\gamma{1+\expo}}$. In facts
\[ \frac{tb'(t)}{b(t)} =  \frac{t}{1+t} \left(-\expo + \frac{\gamma(1+t)}{(c+t)\log(c+t)} \right) < -\expo + \frac\gamma{\log c} <1. \]
In particular, in correspondence with ~$\kappa=0$, we have
\[ \frac{tb'(t)}{b(t)} =  \frac{t\gamma}{(c+t)\log(c+t)} < \frac\gamma{\log c} <1. \]
\end{Example}
\begin{Example}\label{Ex:log2}
Analogously to Example~\ref{Ex:log} we can multiply the function~$b(t)$ in~\eqref{eq:bk} by a logarithmic negative power, namely, we can consider the
coefficient
\begin{equation}
\label{eq:bklogminus} b(t)= \frac\mu{(1+t)^\expo(\log (c+t))^\gamma} \quad \text{for some~$\mu>0,\,\gamma >0$ and~$\expo\in(-1,1)$,}
\end{equation}
where~$c=c(\expo,\gamma)>1$ is a suitably large positive constant. It is easy to check that Hypotheses~\ref{Hyp:b} and~\ref{Hyp:furtherb} are satisfied
if~$c=c(\expo,\gamma)>1$ is sufficiently large.
\end{Example}
\begin{Example} \label{Ex:itelog}
We can also consider iteration of logarithmic functions, eventually with different powers, like
\begin{align*}
%\label{eq:bklog2}
b(t)
    & = \frac\mu{(1+t)^\expo} (\log (c_1+ (\log (c_2 + t))^{\gamma_2}))^{\gamma_1},\\
%\label{eq:bklog3}
b(t)
    & = \frac\mu{(1+t)^\expo} (\log (c_1+ (\log (c_2 + (\log (c_3+\ldots )))^{\gamma_3}))^{\gamma_2})^{\gamma_1}.
\end{align*}
%
%for some~$\mu>0$ and~$\expo\in[-1,1)$, and for some~$\gamma_1, \gamma_2 \ldots,\neq0$. %In particular, if~$\expo=-1$, condition~\eqref{en:bsnotdiss} in Hypothesis~\ref{Hyp:b} holds only if~$\gamma_1>1$ or~$\gamma_1=\ldots=\gamma_l=1$ and~$\gamma_{l+1}>1$ for some~$l\geq1$.
\end{Example}

%%%%%%%%%%%%%%%%%%%%%%%%%%%%%%%%%%%%%%%%%%%%%

\subsection{A special class of effective damping}
In~\cite{N} and~\cite{LNZ} the authors studied damping terms with
time-dependent coefficient ~\eqref{eq:bk}. They can obtain the
following results:
\begin{citThm}\label{Thm:LNZE}
Let $p>p_\Fuj(n)$ and~$p<(n+2)/(n-2)$ if~$n\geq3$. Let
$b(t)=\mu(1+t)^{-\expo}$ for $\expo\in (-1,1)$ and $\mu>0$. Let
$(u_0,u_1)\in H^1\times L^2$, compactly supported.
\\
Then, there exists~$\epsilon_0>0$ such that, if
\begin{equation}\label{eq:LNZdata}
\int_{\R^n} e^{\frac{(1+\expo)|x|^2}{2(2+\delta)}}\left(|u_0(x)|^{p+1}+|\nabla u_0(x)|^2+|u_1(x)|^2\right) dx\leq \epsilon^2
\end{equation}
for an arbitrarily small $\delta > 0$ and for some~$\epsilon\in(0,\epsilon_0]$, then there exists a unique solution $u\in \mathcal{C}([0,\infty), H^1) \cap
\mathcal{C}^1([0,\infty), L^2)$ to \eqref{eq:diss} which satisfies
\begin{align}
\label{eq:estLNZu}
\|u(t,\cdot)\|_{L^2}
    & \leq C(\delta)\epsilon \,(1+t))^{-\frac{(1+\expo)n}4+\frac{\varepsilon}{2}}, \\
\label{eq:estLNZen}
\|\nabla u(t,\cdot)\|_{L^2}+\|u_t(t,\cdot)\|_{L^2}
    & \leq C(\delta)\epsilon \, (1+t)^{-\frac{(1+\expo)(n+2)}4+\frac{\varepsilon}{2}}
\end{align}
for a small constant~$\varepsilon= \varepsilon(\delta) > 0$ and large constant~$C(\delta)$
with~$\varepsilon(\delta)\to 0$ and~$C(\delta)\to \infty$ as~$\delta\to 0$.
\end{citThm}
Moreover, in~\cite{LNZ} the authors establish that there does not exist any global solution $u\in \mathcal{C}([0,\infty),H^1)\cap \mathcal{C}^1([0,\infty),L^2)$
in the case $f(u)=|u|^p$ with $1<p\leq p_\Fuj(n)$ and initial data such that
\[ \int_{\R^n} u_1(x)+\hat b_1 u_0(x)dx>0 \quad \mbox{with} \quad \hat b_1^{-1}=\int_0^\infty \exp \left(-\int_0^t b(s)ds\right)dt. \]
We remark that in~\eqref{eq:LNZdata} the exponents~$p$ and $\expo$
come into play. \\
Recalling Notation~\ref{Not:weight}, for
some~$\beta>0$, $q\geq1$ and~$K>0$, we put
\begin{align*}
D_{\beta, q, K}
    & =\left\{ (u_0,u_1) \in \bigl(\dot H^1_{\beta|x|^2/2}\cap L^q_{\beta|x|^2/q}\bigr)\times L^2_{\beta|x|^2/2} \,|\, \supp (u_0,u_1)\subset B_K(0) \right\},\\
D_{\beta}
    & = H^1_{\beta|x|^2/2}\times L^2_{\beta|x|^2/2}.
\end{align*}
Let~$\beta(\expo,\delta) \doteq (1+\expo)/(2(2+\delta))$. After fixing a small $\delta>0$ the space of initial data in Theorem~\ref{Thm:LNZE} is given by
\[ \bigcup_{K>0} D_{\beta(\expo,\delta), p+1, K},\]%The space of initial data in Theorem~\ref{Thm:LNZE} is given by
%%
%\[ \bigcup_{K>0,\delta>0} D_{\beta(\expo,\delta), p+1, K}\,,\]
%
whereas the space of initial data in Theorem~\ref{Thm:main} is~$D_{2\alpha}$ for some~$\alpha\in(0,1/4]$.
\\
Since~$\beta(\expo,\delta)<1/2$, we observe that for any $\delta>0$,
$p>1$ and $\expo\le 1$ we have
\begin{equation}\label{eq:inclusion}
D_{\beta(\expo,\delta),p+1,K}\subset D_{\beta(\expo,\delta),2,K}\subset D_{1/2,2,K}\subsetneq D_{1/2}\subset D_{2\alpha} \subsetneq \mathcal A_{1,1},
\end{equation}
for any~$K>0$ and~$\alpha\in(0,1/4]$. Hence the class of admissible small data in \cite{LNZ} is strictly contained in the class of admissible small data in
Theorem~\ref{Thm:main}. In particular,
\begin{itemize}
 \item we do not assume compactly supported initial data;
 \item in Theorem \ref{Thm:main} we do not choose $u_0$ from a weighted $L^{p+1}$ space but from a weighted $L^2$ space;
 \item the space with weight $e^{\beta(\expo,\delta)|x|^2}$ is properly contained in $D_{1/2}$, the space in Theorem~\ref{Thm:main} corresponding to $\alpha=1/4$;
 \item in Theorem \ref{Thm:low} we enlarge the class of initial data to ${\mathcal A}_{1,1}$.
\end{itemize}
We can enlarge the class of initial data, since we use Matsumura's type estimates which are avoided in~\cite{LNZ}. This technique has other advantages.
First of all
we can consider more general $b(t)$, not only the ones that growth like $t^\expo$ (see Examples~\ref{Ex:log} and~\ref{Ex:log2},
and Hypothesis~\ref{Hyp:bbs} in Section~\ref{sec:Gen}). \\
Moreover, if $(u_0,u_1)\in D_{\beta(\expo,\delta),p+1, K}$ for some~$K>0$, then applying Theorem~\ref{Thm:LNZE} we know that there exists~$\epsilon_0>0$
such that for any~$\epsilon\in(0,\epsilon_0)$ the solution corresponding to data~$(\epsilon u_0,\epsilon u_1)$ exists globally in time. Here~$\epsilon_0>0$
depends on~$u_0,u_1$ and $K$. Due to~\eqref{eq:inclusion} these data can be used in Theorem~\ref{Thm:main} and Theorem~\ref{Thm:low}, but the corresponding
$\epsilon_0>0$ depends only on $(u_0,u_1)$. Finally, in the decay estimates for the solution~$u$ and the energy~$(\nabla u,u_t)$, an $\varepsilon/2$ loss
of decay appears in Theorem~\ref{Thm:LNZE}, on the contrary, in Theorem \ref{Thm:main} and Theorem \ref{Thm:low} we have optimal decay rates.

%It remains to compare the range of $p$ and $n$ in the above theorems.
%\\
%Theorem~\ref{Thm:low} requires $n\le 4$, and we also lose $p\in(p_\Fuj(n),2)$ for $n=3,4$ in Theorem~\ref{Thm:low},
%due to the $L^1$ estimate of the nonlinear term. But this is natural in the sense that if one enlarges the set of data,
%then one can expect diminishing chances to find a solution.
%\\
%For $n\ge 3$ the range in Theorem~\ref{Thm:LNZE} is larger than the
%range in Theorems~\ref{Thm:main} and~\ref{Thm:low} since the authors
%overcome the $p_\GN(n)$ exponent and they arrive at $p_\Sob(n)$ on
%the level of local existence.

\subsection{Optimality}\label{sec:Optim}
The sharpness of the Fujita exponent $p_{\Fuj}$ in
Theorems~\ref{Thm:main} and~\ref{Thm:low} is of special interest.
This question was discussed by Y. Wakasugi from Osaka University
and the first author during scientific stays at TU Bergakademie
Freiberg. In the following we present only the result. Details of
the proof will be included in a forthcoming paper. \\

If $f(u)=|u|^p$ with $1<p \leq p_{\Fuj}$ blow-up phenomena appear for~\eqref{eq:diss}. This result can be proved by
using, as in \cite{LNZ}, the transformation of the equation into
divergence form and a modification of test function method developed
by Qi. S. Zhang in \cite{Z}.
\\
We make the following assumptions on~$b(t)$.
\begin{Hyp}\label{Hyp:blowup}
Let~$b(t)$ satisfy ~\eqref{en:bpos} and~\eqref{en:1b} in
Hypothesis~\ref{Hyp:b}, that is, $b(t)>0$ and~$1/b \not\in L^1$.
Moreover, we assume that~$b\in\mathcal{C}^2$ and that
\begin{align}
\label{eq:beffblow} |b'(t)|\leq Cb^2(t)
\intertext{together with}
\label{eq:liminfblow} \liminf_{t\to\infty} \frac{b'(t)}{b(t)^2} >-1.
\end{align}
\end{Hyp}
We remark that Hypothesis~\ref{Hyp:blowup} is weaker than Hypothesis~\ref{Hyp:b}.
\begin{citThm}\label{Thm:blowup}
Let us assume Hypotheses~\ref{Hyp:furtherb} and~\ref{Hyp:blowup} and let~$p\leq
p_\Fuj(n)$. Then the function
%
%\begin{equation}\label{eq:beta}
\[ \beta(t)\doteq %\frac1{\lambda^2(t)} =
\exp\left(-\int_0^t
b(\tau)\,d\tau\right)\]
%\end{equation}
is in $L^1(0,\infty)$ and there exists no global
solution~$u\in\mathcal{C}^2([0,\infty)\times\Rn)$ to~\eqref{eq:diss}
with ~$f(u)=|u|^p$ for initial data ~$(u_0,u_1)\in
C_0^{\infty}(\mathbb{R}^n)$ satisfying
\begin{equation}\label{eq:datablow}
\int_{\mathbb{R}^n}\big(u_0(x)+\hat{b}_1u_1(x)\big) dx>0,
\end{equation}
where~$\hat{b}_1\doteq \|\beta\|_{L^1(0,\infty)}^{-1}$.
\end{citThm}
\begin{Example}
Let us consider~$b(t)=\mu (1+t)^{-\expo}$ as in~\eqref{eq:bk} in
Example~\ref{Ex:bk} for some~$\expo\in (-1,1]$ and~$\mu>0$. Then
Hypothesis~\ref{Hyp:blowup} holds provided that~$\mu>1$
if~$\expo=1$.
\\
Let~$b$ be as in~\eqref{eq:bklog} in Example~\ref{Ex:log}, that is,
$b(t)= \mu(1+t)^{-\expo} (\log (c+t))^\gamma$. Then
Hypothesis~\ref{Hyp:blowup} holds for any~$\mu>0, \kappa\in (-1,1],
\gamma>0$ with  a suitable constant~$c$. Analogously, we can prove
Hypothesis~\ref{Hyp:blowup} if~$b$ is chosen as in Examples
\ref{Ex:log2} and \ref{Ex:itelog}.
\end{Example}

%%%%%%%%%%%%%%%%%%%%%%%%%%%%%%%%%%%%%%%%
%
%%%%%%%%%%%%%%%%%%%%%%%%%%%%%%%%%%%%%%%%

\section{Linear decay estimates}

In order to prove Theorems~\ref{Thm:main} and~\ref{Thm:low} we have to extend the decay estimates \eqref{eq:MWu}-\eqref{eq:MWux}-\eqref{eq:MWut} given by
J. Wirth for the Cauchy problem~\eqref{eq:CPlin} to a family of parameter-dependent Cauchy problems with initial data~$(0,g(s,x))$ for some function~$g$.
\\
Let~$s\geq0$ be a parameter. We consider the following Cauchy problem in~$[s,\infty)\times \Rn$:
\begin{equation}\label{eq:CPD}
\begin{cases}
v_{tt}-\triangle v + b(t)v_t=0, \qquad t\in [s,\infty),\\
v(s,x)=0, \\
v_t(s,x)=g(s,x).
\end{cases}
\end{equation}
It is clear that we have to extend Definition~\ref{Def:Bt0}.
\begin{Def}\label{Def:Bts}
We denote by~$B(t,s)$ the primitive of~$1/b(t)$ which vanishes at~$t=s$, that is,
\begin{equation}\label{eq:Bts}
B(t,s) = \int_s^t \frac1{b(\tau)}\, d\tau = B(t,0)-B(s,0).
\end{equation}
\end{Def}
Then we have the following result:
\begin{Thm}\label{Thm:linmain}
Let~$b(t)$ satisfy Hypothesis~\ref{Hyp:b} and let~$g(s,\cdot)\in L^m\cap L^2$ for some~$m\in [1,2]$. Then the solution~$v(t,x)$ to~\eqref{eq:CPD} satisfies
the following Matsumura-type decay estimates$:$
%\footnote{Indeed, it is better to write \eqref{eq:MWDut} in the form
%
%\[
%\|v_t(t,\cdot)\|_{L^2}\leq C (1+b(s)B(t,s))^{-1} (1+B(t,s))^{1-\frac{n}2(\frac1m-\frac12)} \bigl(1+b(t)B(t,s)\bigr)^{-1} \|g(s,\cdot)\|_{L^m\cap L^2}\,.
%\]
%
%In facts, by using this formulation, we avoid the degeneracies which appear in the statement of~\eqref{eq:MWDut} at the initial time~$t=s$, i.e. %$\|g(s,\cdot)\|_{L^2}\leq C (b(s))^{-1} \|g(s,\cdot)\|_{L^m\cap L^2}$, in the case~$b(s)\to\infty$ as~$s\to\infty$.}
%
\begin{align}
\label{eq:MWDu}
\|v(t,\cdot)\|_{L^2}
    & \leq C (b(s))^{-1} (1+B(t,s))^{-\frac{n}2(\frac1m-\frac12)} \|g(s,\cdot)\|_{L^m\cap L^2}, \\
\label{eq:MWDux}
\|\nabla v(t,\cdot)\|_{L^2}
    & \leq C (b(s))^{-1} (1+B(t,s))^{-\frac{n}2(\frac1m-\frac12)-\frac12} \|g(s,\cdot)\|_{L^m\cap L^2}, \\
\label{eq:MWDut}
\|v_t(t,\cdot)\|_{L^2}
    & \leq C (b(s))^{-1} (b(t))^{-1}(1+B(t,s))^{-\frac{n}2(\frac1m-\frac12)-1} \|g(s,\cdot)\|_{L^m\cap L^2}.
\end{align}
We remark that the constant~$C>0$ does not depend on~$s$.
\end{Thm}
We remark that Hypothesis~\ref{Hyp:furtherb} does not come into play in Theorem~\ref{Thm:linmain}.

%%%%%%%%%%%%%%%%%%%%%%%%%%%%%%%%%%%%

\subsection{Application of Duhamel's principle to the semi-linear problem}

Let us denote by~$E_1(t,s,x)$ the fundamental solution to the linear homogeneous problem~\eqref{eq:CPD}, in particular
\[
E_1(s,s,x)=0 \text{ and } \partial_t E_1(s,s,x)=\delta_x,
\]
where~$\delta_x$ is the Dirac distribution in the $x$ variable. Here the symbol~$\ast_{(x)}$ denotes the convolution with respect to the~$x$ variable. By
Duhamel's principle we get
\begin{equation}\label{eq:unl}
u^\nl(t,x)=\int_0^t E_1(t,s,x)\ast_{(x)} f(u(s,x)) \, ds
\end{equation}
as the solution to the inhomogeneous problem
\begin{equation}\label{eq:CPinh}
\begin{cases}
u^\nl_{tt}-\triangle u^\nl + b(t)u^\nl_t=f(u(t,x)), \quad t\in [0,\infty),\\
u^\nl(0,x)=0, \\
u^\nl_t(0,x)=0.
\end{cases}
\end{equation}
Let~$u^\lin(t,x)$ be the solution to~\eqref{eq:CPlin}. Then
\begin{equation}\label{eq:ulin}
u^\lin(t,x)= E_0(t,0,x)\ast_{(x)} u_0(x) + E_1(t,0,x)\ast_{(x)} u_1(x),
\end{equation}
where~$E_1(t,0,x)$ is as above, and by~$E_0(t,0,x)$ we denote the fundamental solution of the homogeneous Cauchy problem~\eqref{eq:CPlin} with initial data
~$(\delta_x,0)$, that is
\[
E_0(0,0,x)=\delta_x \text{ and } \partial_t E_0(0,0,x)=0.
\]
Now the solution to~\eqref{eq:diss} can be written in the form
\begin{align}\nonumber
u(t,x) & = u^\lin(t,x) + u^\nl(t,x) \\
\label{eq:solution}
    & = E_0(t,0,x)\ast_{(x)} u_0(x) + E_1(t,0,x)\ast_{(x)} u_1(x) + \int_0^t E_1(t,s,x)\ast_{(x)} f(u(s,x)) \, ds.
\end{align}

%%%%%%%%%%%%%%%%%%%%%%%%%%%%%%%%%%%%%%%

\subsection{Properties of~$B(t,s)$}

In the proof of Theorems~\ref{Thm:main} and~\ref{Thm:low} we will make use of some properties of the function~$B(t,s)$ which follow from
Hypothesis~\ref{Hyp:furtherb} for the coefficient ~$b(t)$.
\begin{Rem}\label{Rem:solito}
If~\eqref{eq:tbprimeb} holds, then it follows that the function $t/b(t)$ is increasing and
\[ \left(\frac{t}{b(t)}\right)' = \frac{b(t)-tb'(t)}{b^2(t)} \geq (1-m)\, \frac1{b(t)}.\]
Moreover, since~$|b'(t)|/b(t)\leq M/(1+t)$ for some~$M>0$ (see~\eqref{eq:oscb}), we derive
\[ \left(\frac{t}{b(t)}\right)' = \frac{b(t)-tb'(t)}{b^2(t)} \leq \frac{1+M}{b(t)}. \]
In particular, for any~$s\in [0,t]$ we can derive
\begin{equation}\label{eq:Btsbehav2}
B(t,s) = \int_s^t \frac1{b(\tau)}\,d\tau \approx \frac{t}{b(t)} - \frac{s}{b(s)}.
\end{equation}
\end{Rem}
\begin{Rem}
By integrating~\eqref{eq:tbprimeb} over ~$[s,t]$ we derive
\[ \frac{b(t)}{b(s)} \leq \left(\frac{t}{s}\right)^m \quad \text{for any~$s>0$ and~$t\geq s$,} \]
that is, for any~$\homog\in(0,1]$ and for any~$t\in [0,\infty)$, it holds
\begin{equation}
\label{eq:balphac} b(\homog t) \geq \homog^m b(t).
\end{equation}
We remark that, in particular, $b(t)\leq t^mb(1)$ for $t\ge 1$. Therefore Hypothesis~\ref{Hyp:furtherb} implies~\eqref{en:1b} in Hypothesis~\ref{Hyp:b},
since~$m\in [0,1)$.
\end{Rem}
\begin{Rem}
Thanks to~\eqref{eq:oscb} for~$k=1$ there exists a constant ~$M\geq0$ such that
\begin{equation}\label{eq:btprimebC}
\frac{b'(t)}{b(t)} \geq -\frac{M}{1+t} \geq -\frac{M}t, \quad t>0.
\end{equation}
It is clear that if~$b(t)$ is increasing, then we can take~$M=0$.
\\
By integrating~\eqref{eq:btprimebC} over ~$[s,t]$ we derive
\[ \frac{b(t)}{b(s)} \geq \left(\frac{t}{s}\right)^{-M} \quad \text{for any~$s>0$ and~$t\geq s$,} \]
that is, for any~$\homog\in(0,1]$ and for any~$t\in [0,\infty)$ it holds
\begin{equation}
\label{eq:balphaC} b(\homog t) \leq \homog^{-M} b(t).
\end{equation}
\end{Rem}
Properties \eqref{eq:balphac}-\eqref{eq:balphaC} play a fundamental role in the next estimates.
\begin{Rem}
Conditions \eqref{eq:balphac}-\eqref{eq:balphaC} guarantee that for any fixed~$\homog\in(0,1)$ we have
\begin{equation}\label{eq:bsbt}
b(s)\approx b(t), \quad s\in [\homog t,t].
\end{equation}
Indeed, let~$\homog_1\doteq s/t$. Then ~$\homog_1\in [\homog,1]$. Hence, we get
\[ \homog^mb(t)\leq \homog_1^m b(t)\le b(s)\le \homog_1^{-M}b(t)\leq \homog^{-M}b(t) \]
from \eqref{eq:balphac}-\eqref{eq:balphaC}.
\end{Rem}
\begin{Rem}
By using~\eqref{eq:tbprimeb} and its consequences~\eqref{eq:Btsbehav2} and \eqref{eq:balphac} we can prove that for any fixed~$\homog\in(0,1)$ it holds
\[ B(t,0)\geq B(t,\homog t) \approx \frac{t}{b(t)}-\frac{\homog t}{b(\homog t)} \geq \frac{t}{b(t)}-\frac{\homog^{1-m} t}{b(t)}
= \delta \frac{t}{b(t)} \approx B(t,0),\]
where we put ~$\delta = 1-\homog^{1-m}>0$ since~$\homog\in (0,1)$ and ~$m\in [0,1)$. Therefore,
\begin{equation}
\label{eq:Btalphat} C_{\homog,m} B(t,0) \leq B(t,\homog t)\leq B(t,0)\quad \text{for~$\homog\in(0,1)$.}
\end{equation}
\end{Rem}
\begin{Rem}
By using~\eqref{eq:Btsbehav2} and~\eqref{eq:balphaC} we can prove that for any fixed ~$\homog\in(0,1)$ it holds
\[ B(\homog t,0) \approx \frac{\homog t}{b(\homog t)} \geq \homog^{1+M} \frac{t}{b(t)}, \]
and, consequently,
\begin{equation}
\label{eq:Balphat0} C_{\homog,M} B(t,0) \leq B(\homog t,0)\leq B(t,0) \quad \text{for~$\homog\in(0,1)$.}
\end{equation}
\end{Rem}
\begin{Rem}
By splitting the interval $[0,t]$ into~$[0,t/2]$ and $[t/2,t]$ and by using~\eqref{eq:Balphat0} we can derive
\begin{align}
\label{eq:Bs0large} B(s,0) \approx B(t,0), & \quad s\in [t/2,t],
\intertext{whereas by using~\eqref{eq:Btalphat} we get}
\label{eq:Btssmall} B(t,s) \approx B(t,0), & \qquad s\in [0,t/2].
\end{align}
\end{Rem}
\begin{Rem}
By using Taylor-Lagrange's theorem (with center~$t$) and~\eqref{eq:bsbt} with~$\homog=1/2$ we obtain
\begin{equation}
\label{eq:Btslarge} B(t,s) \approx \frac{t-s}{b(t)} \approx \frac{t-s}{b(s)}, \quad s\in [t/2,t].
\end{equation}
Indeed $b(s)\approx b(r)\approx b(t)$ for any $r\in [s,t]\subset [t/2,t]$, thanks to~\eqref{eq:bsbt}, and
\[ B(t,s)=B(t,t) + (s-t)\partial_s B(t,r) = 0 + \frac{t-s}{b(r)}  \quad \text{for some~$r\in [s,t]$.} \]
\end{Rem}
\begin{Rem}\label{Rem:bminus1B}
We observe that
\[ b(t)(1+B(t,0)) \approx 1+b(t)B(t,0)\approx 1+t.\]
Thanks to~\eqref{eq:Btsbehav2} it suffices to prove only the first equivalence.
\\
Since~$b(t)>0$ for any~$t>0$, the equivalence holds on compact intervals. It remains to observe that the behavior of the two objects is described in both cases by $b(t)B(t,0)$ for~$t\to\infty$.

Indeed, since $B(t,0)\to \infty$ (we recall that $1/b\not\in L^1$), it follows $1+B(t,0) \approx B(t,0)$, therefore $b(t)(1+B(t,0))\approx b(t)B(t,0)$. On the other hand, applying once more
~\eqref{eq:Btsbehav2}, it follows $b(t)B(t,0)\geq C\,t \to \infty$. Therefore $1+b(t)B(t,0)\approx
b(t)B(t,0)$.
\end{Rem}
%We observe that for any fixed~$s\geq0$ and for any~$t\in [s,\infty)$, it holds
%
%\[ b(t)(1+B(t,s)) \approx 1+b(t)B(t,s)\,.\]
%
%Indeed, for~$t=s$, the first term is equal to~$b(s)>0$ and the second one to~$1$. On the other hand, the increasing behavior of the two objects as~$t\to\infty$ %is described in both cases by $b(t)B(t,s)$. Indeed, since $B(t,s)\to+\infty$ (we recall that $1/b\not\in L^1$ is positive), it follows $1+B(t,s) \approx %B(t,s)$, therefore $b(t)(1+B(t,s))\approx b(t)B(t,s)$. On the other hand, thanks to~\eqref{eq:Btssmall} and~\eqref{eq:Btsbehav2}, it follows $b(t)B(t,s)\approx %b(t)B(t,0) \geq C\,t \to+\infty$ as~$t\to\infty$, therefore $1+b(t)B(t,s)\approx b(t)B(t,s)$ as~$t\to\infty$.
%\\
%Analogously, for any fixed~$t\geq0$ and for any~$s\in [0,t]$, it holds
%
%\[b(s)(1+B(t,s)) \approx 1+b(s)B(t,s)\,.\]
%
%\end{Rem}

%%%%%%%%%%%%%%%%%%%%%%%%%%%%%%%%%%%
%
%%%%%%%%%%%%%%%%%%%%%%%%%%%%%%%%%%%

\section{Proof of Theorem~\ref{Thm:main}}\label{sec:proofmain}

\subsection{Local existence in weighted energy space}\label{sec:localweight}

We have the following local existence result in weighted energy spaces.
\begin{Lem}\label{Lem:localweight}
Let~$b(t)>0$. Let $1<p\le p_\GN(n)$. Let $\psi\in \mathcal C^1([0,\infty)\times \R^n)$ such that for any $t\ge 0$ and $a.e.\; x\in \R^n$ one has
\begin{equation}\label{psi}
\begin{array}{l}
        \psi(t,x)\geq0,\\
   \psi_t(t,x)\le 0,\\
   b(t)\psi_t(t,x)+|\nabla \psi(t,x)|^2\le 0,\\
    \Delta \psi(t,x)> 0,\\
 \inf_{x\in \R^n} \Delta\psi(t,x)=C(t)>0.
\end{array}
\end{equation}
For any $(u_0,u_1)\in H^1(e^{\psi(0,x)})\times L^2(e^{\psi(0,x)})$ there exists a maximal existence time $T_m\in(0,\infty]$ such that \eqref{eq:diss} has a
unique solution
 $u\in\mathcal{C}([0,T_m),H^1)\cap\mathcal{C}^1([0,T_m),L^2)$. Moreover, for any $T<T_m$ it holds
\[ \sup_{[0,T]}\| e^{\psi(t,\cdot)}u(t,\cdot)\|_{L^2}+\| e^{\psi(t,\cdot)}\nabla u(t,\cdot)\|_{L^2}+\| e^{\psi(t,\cdot)}u_t(t,\cdot)\|_{L^2} <\infty.\]
Finally, if $T_m<\infty$, then
\begin{equation}\label{eq:blowup}
\limsup_{t\to T_m}\| e^{\psi(t,\cdot)}u(t,\cdot)\|_{L^2}+\| e^{\psi(t,\cdot)}\nabla u(t,\cdot)\|_{L^2}+\| e^{\psi(t,\cdot)}u_t(t,\cdot)\|_{L^2} =\infty.
\end{equation}
\end{Lem}
The proof follows the same lines of the Appendix of~\cite{IT}. We underline that the local existence result does not require Hypotheses~\ref{Hyp:b} or~\ref{Hyp:furtherb}.

%%%%%%%%%%%%%%%%%%%%%%%%%%%%%%%%%%%%%%%%%%%%%%%

\subsection{Energy estimates in weighted energy space}

Let us observe that the function $\psi(t,x)$ given in~\eqref{eq:psi} satisfies~\eqref{psi} since~$\alpha\in(0,1/4]$. Therefore the local existence result is applicable. Indeed
\[ \psi(t,x)=\frac{\alpha|x|^2}{1+B(t,0)} \]
verifies
\[ \psi_t=-\frac{\alpha |x|^2}{(1+B(t,0))^2b(t)}, \quad \nabla\psi=\frac{2\alpha x}{1+B(t,0)}, \quad \triangle\psi =\frac{2n\alpha }{1+B(t,0)}, \]
together with the fundamental property
\begin{equation}\label{eq:weightfund}
b(t)\psi_t+|\nabla\psi|^2 = -\frac{\alpha(1-4\alpha) |x|^2}{(1+B(t,0))^2} \leq 0
\end{equation}
since~$\alpha\in(0,1/4]$. We underline that for $\alpha=1/4$ the equation~$b(t)\psi_t+|\nabla\psi|^2=0$ is related to the symbol of the linear parabolic equation~$b(t)u_t-\triangle u=0$, that is, we have in mind the \emph{parabolic effect} when we introduce the weight~$e^{\psi(t,x)}$.
\begin{Lem}\label{Lem:Erefined}
Let us assume that $(u_0,u_1)\in H^1_{\psi(0,x)}\times L^2_{\psi(0,x)}$, and let~$\gamma=2/(p+1)+\varepsilon$ for some~$\varepsilon>0$. If $u=u(t,x)$ is a local solution to the equation in ~\eqref{eq:diss} in $[0,T)$, then for any $t\in [0,T)$ the following energy estimate holds$:$
\begin{equation}\label{eq:psiE}
E(t) \leq C I_\alpha^2 + CI_\alpha^{p+1} + C_\varepsilon \left( \sup_{[0,t]} (1+B(s,0))^\varepsilon \|e^{\gamma\psi(s,\cdot)}u(s,\cdot)\|_{L^{p+1}}\right)^{p+1}\,,
\end{equation}
with~$I_\alpha$ given by~\eqref{eq:I0alpha} and
\[ E(t)\doteq\frac{1}{2} \int_{\R^n} e^{2\psi(t,x)}\left(|u_t(t,x)|^2+|\nabla u(t,x)|^2\right) dx\,.\]
\end{Lem}
\begin{proof}
{}First we prove that
\begin{equation}\label{eq:Estart}
E(t) \lesssim I_\alpha^2 + I_\alpha^{p+1} + \|e^{\frac2{p+1}\psi(t,\cdot)}u(t,\cdot)\|_{L^{p+1}}^{p+1} + \int_0^t\int_{\Rn} |\psi_t(s,x)| e^{2\psi(s,x)}
|u(s,x)|^{p+1} dx ds.
\end{equation}
Straight-forward calculations gives the following relation:
\begin{multline*}
\partial_t \left( \frac{e^{2\psi}}2\, \bigl(|u_t|^2+|\nabla u|^2-F(u)\bigr) \right) \\
    = \nabla\cdot(e^{2\psi}u_t\nabla u) +\psi_te^{2\psi}|u_t|^2+\frac{e^{2\psi}}
    {\psi_t}|u_t\nabla\psi-\psi_t\nabla u|^2-\frac{e^{2\psi}}{\psi_t}u_t^2(b(t)\psi_t+|\nabla\psi|^2)-2\psi_te^{2\psi}F(u),
\end{multline*}
where $F(u)\doteq\int_0^u f(\tau) d\tau$  is a primitive of the nonlinear term $|f(\tau)|\simeq |\tau|^p$, hence, $|F(u)|\leq C |u|^{p+1}$.
\\
After integration over~$[0,t]\times\Rn$, by taking into consideration $\psi_t\leq0$ and \eqref{eq:weightfund}  we can estimate
\[ G(t)\leq G(0)-2\int_0^t\int_{\Rn} \psi_t(s,x) e^{2\psi(s,x)} F(u(s,x)) dxds, \]
where we put
\[ G(t)\doteq E(t)- \int_{\Rn} \frac{e^{2\psi(t,x)}}2\, F(u(t,x)) dx = \int_{\Rn} \frac{e^{2\psi(t,x)}}2\, \Big(|u_t(t,x)|^2+|\nabla u(t,x)|^2-F(u(t,x))\Big) dx.
\]
We remark that the divergence theorem can be applied being
\[ e^{2\psi(s,\cdot)}u_t(s,\cdot)\nabla u(s,\cdot)\in L^1(\R^n).\]
This follows from Lemma~\ref{Lem:localweight}. Therefore,
\[ E(t) \lesssim G(0) + \|e^{\frac2{p+1}\psi(t,\cdot)}u(t,\cdot)\|_{L^{p+1}}^{p+1} + \int_0^t\int_{\Rn}
|\psi_t(s,x)| e^{2\psi(s,x)} |u(s,x)|^{p+1} dxds.
\]
In order to gain \eqref{eq:Estart} it remains to show that $G(0)\lesssim I_\alpha^2+I_\alpha^{p+1}$. This reduces to prove that
\[ \int_{\R^n} e^{2\alpha |x|^2} |u_0|^{p+1} dx \lesssim  I_\alpha^{p+1}. \]
Since $p+1<p_{GN}(n)+1 \le \frac{2n}{n-2}$ for $n\ge 3$ (no requirement for $n=1,2$) from Sobolev embedding it follows that
\[
\int_{\R^n} e^{2\alpha |x|^2} |u_0|^{p+1} dx \lesssim \left[ \int_{\R^n} e^{\frac{4\alpha}{p+1} |x|^2} \left(|u_0|^{2} +|\nabla u_0|^2\right)dx
+\int_{\R^n}e^{\frac{4\alpha}{p+1} |x|^2}|x|^2 |u_0|^{2} dx \right]^{\frac{p+1}2}\,.
\]
The assumption $p>1$ gives $(1+|x|^2) e^{\frac{4\alpha}{p+1} |x|^2} \leq C e^{2\alpha |x|^2}$. This  concludes the proof of~\eqref{eq:Estart}.
\\
Now, by virtue of
\[ |\psi_t(s,x)| e^{(2-\gamma(p+1))\psi(s,x)} = \frac{\psi(s,x)}{(1+B(s,0))b(s)}  e^{-(p+1)\varepsilon\psi(s,x)} \leq \frac{C_\varepsilon}{(1+B(s,0))b(s)} \]
from~\eqref{eq:Estart} we derive
\[ E(t) \leq C I_\alpha^2 + C I_\alpha^{p+1} + C \|e^{\frac2{p+1}\psi(t,\cdot)}u(t,\cdot)\|_{L^{p+1}}^{p+1}
+ C_\varepsilon \int_0^t \frac1{(1+B(s,0))b(s)} \|e^{\gamma\psi(s,x)} u(s,x)\|_{L^{p+1}}^{p+1} ds.\]
For any~$\varepsilon>0$ it holds
\[ \int_0^t \frac1{(1+B(s,0))^{1+\varepsilon}b(s)}\,ds = \int_1^{1+B(t,0)} \frac1{\tau^{1+\varepsilon}}\,d\tau\leq \frac1\varepsilon\,, \]
therefore
\[ E(t) \leq C I_\alpha^2 + CI_\alpha^{p+1} + C \|e^{\frac2{p+1}\psi(t,\cdot)}u(t,\cdot)\|_{L^{p+1}}^{p+1} + C_\varepsilon' \left( \sup_{[0,t]} (1+B(s,0))^\varepsilon \|e^{\gamma\psi(s,\cdot)}u(s,\cdot)\|_{L^{p+1}}\right)^{p+1}\,. \]
To complete the proof it is sufficient to notice that the third term is estimated by the fourth one, since~$\gamma>2/(p+1)$ and~$B(s,0)\geq0$.
\end{proof}

%%%%%%%%%%%%%%%%%%%%%%%%%%%%%%%%%%%%%%%%%

\subsection{Decay estimates for the semi-linear problem}

Let us observe that we can apply the estimates in Theorem~\ref{Thm:linmain} for~$m=1$ if~$(u_0,u_1)\in H^1_{\alpha|x|^2}\times L^2_{\alpha|x|^2}$. Indeed,
for any ~$v\in L^2_{\alpha|x|^2}$ it holds
\[
\int_{\R^n} |v(x)|\,dx \leq \left( \int_{\R^n} e^{2\alpha|x|^2} |v(x)|^2 \,dx \right)^{\frac12} \left( \int_{\R^n} e^{-2\alpha|x|^2} \,dx
\right)^{\frac12}.
\]
Hence,
\begin{equation}\label{eq:embed}
H^1_{\alpha|x|^2}\times L^2_{\alpha|x|^2}\subset (W^{1,1}\cap H^1)\times (L^1\cap L^2)\subset \mathcal{A}_{1,1}\,.
\end{equation}
Having in mind the application of Theorem~\ref{Thm:linmain} for~$m=1$ we need to estimate~$f(u(s,\cdot))$ in $L^1\cap L^2$ by using the weighted energy spaces. \\
Analogously to Lemma 2.5 in ~\cite{IT}, after a change of variables one has for any $\beta\ge 0$
\[ \int_{\Rn} e^{-\frac{\beta|x|^2}{(1+B(t,0))}} dx = \left(\frac{1+B(t,0)}{\beta}\right)^{n/2} \int_{\Rn} e^{-|y|^2} dy \leq C_{\beta} (1+B(t,0))^{n/2}. \]
Applying H\"older's inequality this implies for any $\varepsilon>0$ it holds
\begin{equation}\label{eq:fus}
\|f(u(s,\cdot))\|_{L^1}\leq C \|u(s,\cdot)\|_{L^p}^p \leq C_{\varepsilon,p} (1+B(s,0))^{n/4} \|e^{\varepsilon\psi(s,\cdot)}u(s,\cdot)\|_{L^{2p}}^p.
\end{equation}
On the other hand, by using the trivial estimate $\|e^{-2\varepsilon p\psi(t,\cdot)}\|_{L^\infty}\leq C$ we get
\begin{equation}\label{eq:fus2}
\|f(u(s,\cdot))\|_{L^2}\leq C \|e^{\varepsilon\psi(s,\cdot)}u(s,\cdot)\|_{L^{2p}}^p.
\end{equation}
Thanks to Theorem~\ref{Thm:linmain} combined with the estimates \eqref{eq:fus}-\eqref{eq:fus2} we are able to prove the following fundamental statement,
which is completely analogous to Lemma 2.4 in~\cite{IT} for~$b\equiv1$.
\begin{Lem}\label{Lem:fundamental}
For~$j+l=0,1$ it holds
\[(b(t))^l(1+B(t,0))^{(n/4+j/2)+l} \| \nabla^j\partial_t^l u(t,\cdot)\|_{L^2} \leq C I_\alpha
+ C_\varepsilon \left(\sup_{[0,t]} h(s) \|e^{\varepsilon\psi(s,\cdot)} u(s,\cdot)\|_{L^{2p}} \right)^p, \]
where we put
\begin{equation}\label{eq:hs}
h(s) \doteq (1+B(s,0))^{\frac{n/4+1+\varepsilon}p}.
\end{equation}
\end{Lem}
\begin{proof}
We come back to the representation of the solution to~\eqref{eq:diss} given in~\eqref{eq:solution}. Recalling~\eqref{eq:embed} it holds $\|(u_0,u_1)\|_{\mathcal A_{1,1}}\le C I_\alpha$. Thanks to~\eqref{eq:MWu} and~\eqref{eq:MWux} for~$m=1$ and $j=0,1$, we get
\[ \|\nabla^j u^\lin(t,\cdot)\|_{L^2} \leq C I_\alpha (1+B(t,0))^{-(n/4+j/2)},\]
and thanks to~\eqref{eq:MWut} for~$m=1$ we derive
\[  \| \partial_t u^\lin(t,\cdot)\|_{L^2} \leq C (b(t))^{-1}I_\alpha (1+B(t,0))^{-n/4-1}. \]
Therefore, we can focus our attention to the nonlinear contribution
\[ u^\nl (t,x) = \int_0^t E_1(t,s,x)\ast f(u(s,x))\,ds. \]
We first consider~$s\in [0,t/2]$. If~$s\in [0, t/2]$, then property~\eqref{eq:Btssmall} gives us $B(t,s)\approx B(t,0)$.
Therefore, thanks to~\eqref{eq:MWDu} and~\eqref{eq:MWDux}, by using~\eqref{eq:fus} and~\eqref{eq:fus2}, we estimate
\begin{multline*}
\Big\| \nabla^j \int_0^{t/2} E_1(t,s,x)\ast f(u(s,x))\, ds \Big\|_{L^2} \\
\leq C \int_0^{t/2} (b(s))^{-1} (1+B(t,s))^{-(n/4+j/2)} (1+B(s,0))^{n/4} \|e^{\varepsilon\psi(s,\cdot)}u(s,\cdot)\|_{L^{2p}}^p ds \\
\leq C (1+B(t,0))^{-(n/4+j/2)} \left(\sup_{[0,t]} h(s) \|e^{\varepsilon\psi(s,\cdot)} u(s,\cdot)\|_{L^{2p}} \right)^p \int_0^{t/2}
(b(s))^{-1}(1+B(s,0))^{-(1+\varepsilon)} ds.
\end{multline*}
After the change of variables ~$r=B(s,0)$ we derive
\begin{equation}
\label{eq:intsmall} \int_0^{t/2} \frac1{b(s)} (1+B(s,0))^{-(1+\varepsilon)} ds = \int_0^{B(t/2,0)} (1+r)^{-(1+\varepsilon)} dr \leq C_\varepsilon.
\end{equation}
Since~$E_1(t,t,x)=0$ for any~$t\in [0,\infty)$ we remark that
\[ \partial_t u^\nl (t,x) = \int_0^t \partial_t E_1(t,s,x)\ast f(u(s,x))\,ds. \]
Taking into consideration~\eqref{eq:MWDut}, \eqref{eq:fus}, \eqref{eq:fus2} and~\eqref{eq:intsmall} we have
\begin{multline*}
\Big\| \int_0^{t/2} \partial_t E_1(t,s,x)\ast f(u(s,x))\, ds \Big\|_{L^2} \\
\leq C \int_0^{t/2} (b(s)b(t))^{-1} (1+B(t,s))^{-n/4-1} (1+B(s,0))^{n/4} \|e^{\varepsilon\psi(s,\cdot)}u(s,\cdot)\|_{L^{2p}}^p ds \\
\leq C (b(t))^{-1}(1+B(t,0))^{-n/4-1} \left(\sup_{[0,t]} h(s) \|e^{\varepsilon\psi(s,\cdot)} u(s,\cdot)\|_{L^{2p}} \right)^p.
\end{multline*}
Now we consider~$s\in [t/2,t]$. Formula~\eqref{eq:Bs0large} gives us~$B(s,0)\approx B(t,0)$. On the other hand~\eqref{eq:Btslarge} gives us~$B(t,s)\approx
(t-s)/b(t)$. It is sufficient to use the energy estimates (that is, the $L^2-L^2$ theory for the linear Cauchy problem given by
\eqref{eq:MWDu}-\eqref{eq:MWDux}-\eqref{eq:MWDut} with $m=2$):
\[ \| \nabla^j \partial_t^l E_1(t,s,x)\ast f(u(s,x)) \|_{L^2} \lesssim (b(s))^{-1}(b(t))^{-l} (1+B(t,s))^{-j/2-l} \|u(s)\|_{L^{2p}}^p, \]
that holds for $j+l=0,1$. Therefore, it follows
\begin{multline*}
\Big\| \int_{t/2}^t \nabla^j \partial_t^l E_1(t,s,x)\ast f(u(s,x))\, ds \Big\|_{L^2} \\
\leq C \left(\sup_{[0,t]} h(s) \|e^{\varepsilon\psi(s,\cdot)} u(s,\cdot)\|_{L^{2p}} \right)^p (h(t/2))^{-p} \frac1{(b(t))^l} \int_{t/2}^t \frac1{b(s)}
(1+B(t,s))^{-j/2-l} ds.
\end{multline*}
For~$j=0$ and~$l=0$ we derive
\begin{equation}\label{eq:intlarge1}
\int_{t/2}^t \frac1{b(s)} ds = B(t,t/2) \leq 1+B(t,0),
\end{equation}
whereas for~$j=1$ and~$l=0$ after putting ~$r=B(t,s)$ we conclude
\begin{equation}\label{eq:intlargej}
\int_{t/2}^t \frac1{b(s)} (1+B(t,s))^{-1/2}ds = \int_0^{B(t,t/2)} (1+r)^{-1/2} dr = 2(1+B(t,t/2))^{1/2}-2 \lesssim (1+B(t,0))^{1/2},
\end{equation}
and, analogously, for~$j=0$ and~$l=1$ we obtain
\begin{equation}\label{eq:intlargel}
\int_{t/2}^t \frac1{b(s)} (1+B(t,s))^{-1}ds = \int_0^{B(t,t/2)} (1+r)^{-1} dr = \log(1+B(t,t/2))\leq \log(1+B(t,0)).
\end{equation}
To conclude the proof it is sufficient to notice that
\[ (h(t/2))^{-p} (b(t))^{-l} (1+B(t,0))^{1-j/2-l}(\log(1+B(t,0))^l \lesssim (b(t))^{-l}(1+B(t,0))^{-n/4-j/2-l}, \]
for~$j+l=0,1$.
\end{proof}

%%%%%%%%%%%%%%%%%%%%%%%%%%%%%%%%%%%%%%%%%%%%%

\subsection{Conclusion of the proof to Theorem~\ref{Thm:main}}

Let us define
\begin{align*}
W(\tau)
    & \doteq \|e^{\psi(\tau,\cdot)}(\partial_t,\nabla)u(\tau,\cdot)\|_{L^2} + (1+B(\tau,0))^{(n/4+1/2)}\|\nabla u(\tau,\cdot)\|_{L^2}\\
    & \qquad +b(\tau)(1+B(\tau,0))^{n/4+1}\|u_t(\tau,\cdot)\|_{L^2} + (1+B(\tau,0))^{n/4}\|u(\tau,\cdot)\|_{L^2}.
\end{align*}
Thanks to Lemmas~\ref{Lem:Erefined} and~\ref{Lem:fundamental} we can estimate
\[ \sup_{[0,t]} W(\tau) \lesssim %I_0
I_\alpha+ I_\alpha^{\frac{p+1}2} + \bigl( \sup_{\tau\in [0,t]} (1+B(\tau,0))^\varepsilon \|e^{\gamma\psi(\tau,\cdot)}u(\tau,\cdot)\|_{L^{p+1}}
\bigr)^{(p+1)/2}+\bigl(\sup_{\tau\in [0,t]} h(\tau) \|e^{\varepsilon\psi(\tau,\cdot)}u(\tau,\cdot)\|_{L^{2p}}\bigr)^p.
\]
In order to manage the last two terms we use a Gagliardo-Nirenberg type inequality (see Lemma~\ref{Lem:GNweight} in Appendix~\ref{sec:GN}) and we get
\begin{equation}\label{eq:GNweight}
\|e^{\sigma\psi(t,\cdot)}v\|_{L^q} \leq C_\sigma (1+B(t,0))^{(1-\theta(q))/2}\, \|\nabla v\|^{1-\sigma}_{L^2}\,\|e^{\psi(t,\cdot)}\nabla v\|^\sigma_{L^2}
\end{equation}
for any~$\sigma\in [0,1]$ and~$v\in H^1_{\sigma\psi(t,\cdot)}$, where
\begin{equation}\label{eq:thetaq}
\theta(q)\doteq \frac{n}2-\frac{n}q = n \Big(\frac12-\frac1q\Big)
\end{equation}
for~$q\geq2$, together with $q\leq2^*$ if~$n\geq3$, where~$2^*\doteq 2n/(n-2)=2p_\GN(n)$.
\\
By using~\eqref{eq:GNweight}, since~$\gamma=2/(p+1)+\varepsilon$, it follows
\begin{align}
\label{eq:GNp1}
\|e^{\gamma\psi(\tau,\cdot)}u(\tau,\cdot)\|_{L^{p+1}} & \leq W(\tau)\, (1+B(\tau,0))^{(1-\theta(p+1))/2-(1-2/(p+1)-\varepsilon)(n/4+1/2)},\\
\label{eq:GN2p}
\|e^{\varepsilon\psi(\tau,\cdot)}u(\tau,\cdot)\|_{L^{2p}} & \leq W(\tau)\, (1+B(\tau,0))^{((1-\theta(2p))/2-(1-\varepsilon)(n/4+1/2))}.
\end{align}
Recalling~\eqref{eq:hs}, we observe that the quantities
\begin{align}
\label{eq:W1}
\max_{\tau\in[0,t]} (1+B(\tau,0))^{\frac{1-\theta(p+1)}2 - \left(1-\frac2{p+1}-\varepsilon\right)\left(\frac{n}4+\frac12\right)+\varepsilon},\\
\label{eq:W2} \max_{\tau\in[0,t]} (1+B(\tau,0))^{\frac{n/4+1+\varepsilon}{p}+\frac{1-\theta(2p)}2-(1-\varepsilon)(\frac{n}4+\frac12)},
\end{align}
are uniformly bounded in~$[0,\infty)$, provided that~$\varepsilon>0$ is sufficiently small, since~$p>p_\Fuj(n)$. Indeed,
\[ \frac{1-\theta(p+1)}2 - \Big(1-\frac2{p+1}\Big)\Big(\frac{n}4+\frac12\Big) = \frac{n/4+1}{p}+\frac{1-\theta(2p)}2
-\Big(\frac{n}4+\frac12\Big) = \frac{1-(p-1)n/2}{p} <0. \]
Let us define
\[ M(t)\doteq\max_{[0,t]}W(\tau), \]
and let~$\epsilon=I_\alpha$. We remark that $M(0)=W(0)\leq (2+b(0))\epsilon$. We have proved that
\begin{equation}\label{eq:Mconcav}
M(t)\leq c_0(\epsilon + \epsilon^{p+1}) + c_1(M(t))^{\frac{p+1}2} + c_2(M(t))^p
\end{equation}
for some~$c_0,c_1,c_2>0$. We claim that there exists a constant~$\epsilon_0>0$ such that for any~$\epsilon\in(0,\epsilon_0]$ it holds
\begin{equation}
\label{eq:Mtbound} M(t)\leq C\epsilon,
\end{equation}
in particular $E(t)\le C^2 \epsilon^2$, uniformly with respect to~$t\in[0,\infty)$. Straightforward calculations (see~\cite{IT}) give also
\begin{equation}\label{eq:epsiu}
\|e^{\psi(t,\cdot)}u(t,\cdot) \|_{L^2} \lesssim \epsilon (1+t), \quad t\in [0,T).
\end{equation}
Thanks to~\eqref{eq:Mtbound} and~\eqref{eq:epsiu}, the global existence of the solution follows by contradiction with the condition~\eqref{eq:blowup} of
Lemma~\ref{Lem:localweight}. Let us prove our claim~\eqref{eq:Mtbound}. We define
\[ \phi(x) = x -c_1x^{\frac{p+1}2} -c_2 x^p \]
for some fixed constants~$c_1,c_2>0$. We notice that~$\phi(0)=0$
and~$\phi'(0)=1$. Moreover, $\phi(x)\leq x$ for any~$x\geq0$, and we
take~$\overline{x}>0$ such that~$\phi'(x)\geq 1/2$
on~$[0,\overline{x}]$. Therefore~$\phi$ is strictly increasing
and~$\phi(x)\leq x\leq2\phi(x)$ for any~$x\in[0,\overline{x}]$. Let
\[ \epsilon_0\doteq \min \left\{1, \frac{\overline{x}}{2+b(0)}, \frac{\overline{x}}{4c_0}\right\}. \]
If~$I_\alpha=\epsilon$ for some~$\epsilon\in(0,\epsilon_0]$, then
\begin{equation}\label{eq:M0x}
M(0)=W(0)\leq (2+b(0))\epsilon<\overline{x}.
\end{equation}
Since~$\phi(x)$ is strictly increasing on~$[0,\overline{x}]$ it
follows from~\eqref{eq:M0x} that
\begin{equation}\label{eq:phiM0}
\phi(M(0))\leq\phi(\overline{x}).
\end{equation}
Thanks to~\eqref{eq:Mconcav} we get
\begin{equation}\label{eq:Mt}
\phi(M(t))\leq c_0(\epsilon+\epsilon^p)\leq 2c_0\epsilon
\end{equation}
for any~$t\geq0$. Since~$M(t)$ is a continuous function and
\[ 2c_0\epsilon<2c_0\epsilon_0\leq \overline{x}/2\leq \phi(\overline{x}) \]
it follows from~\eqref{eq:phiM0} and~\eqref{eq:Mt}
that~$M(t)\in(0,\overline{x})$ for any~$t\geq0$. Therefore,
since~$x\leq2\phi(x)$ in~$[0,\overline{x}]$ we also derive
from~\eqref{eq:Mt} that
\[ M(t)\leq 2\phi(M(t))\leq 4c_0\epsilon.\]
This concludes the proof of~\eqref{eq:Mtbound} and as a consequence
the global existence result. The relation~\eqref{eq:Mtbound} implies
directly the decay estimates
\eqref{eq:decayu}-\eqref{eq:decayux}-\eqref{eq:decayut} for the
semi-linear problem~\eqref{eq:diss} (see Remark~\ref{Rem:bminus1B}).

%%%%%%%%%%%%%%%%%%%%%%%%%%%%%%%%%%%%%%%%%%%%%
%
%%%%%%%%%%%%%%%%%%%%%%%%%%%%%%%%%%%%%%%%%%%%%

\section{Proof of Theorem~\ref{Thm:low}}

%\subsection{An auxiliary operator $N$}

In order to prove the global existence of a solution
in~$\mathcal{C}([0,\infty),H^1)\cap\mathcal{C}^1([0,\infty),L^2)$
such that the estimates
\eqref{eq:decayulow}-\eqref{eq:decayuxlow}-\eqref{eq:decayutlow} are
satisfied for any~$t\geq0$ we introduce the space
\[ X(t)=\left\{ u\in \mathcal{C}([0,t],H^1)\cap \mathcal{C}^1([0,t],L^2) \right\} \]
with the norm
\begin{align*}
\|u\|_{X(t)} \doteq \sup_{0\leq \tau\leq t} \bigl[
    & (1+B(\tau,0))^{n/4} \|u(\tau,\cdot)\|_{L^2} +(1+B(\tau,0))^{n/4+1/2}\|\nabla u(\tau,\cdot)\|_{L^2} \\
    & +(1+B(\tau,0))^{n/4}(1+\tau) \|u_t(\tau,\cdot)\|_{L^2} \bigr].
\end{align*}
We remark that if~$u\in X(t)$, then~$\|u\|_{X(s)}\leq\|u\|_{X(t)}$ for any~$s\leq t$. %Moreover, if~$u\in\mathcal{C}([0,\infty),H^1)\cap\mathcal{C}^1([0,\infty),L^2)$, then~$u\in X(t)$ for any~$t\geq0$.
\\
We shall prove that for any data~$(u_0,u_1)\in \mathcal{A}_{1,1}$
the operator~$N$ which is defined by
\[ Nu(t,x) = E_0(t,0,x)\ast_{(x)} u_0(x) + E_1(t,0,x)\ast_{(x)} u_1(x) + \int_0^t E_1(t,s,x)\ast_{(x)} f(u(s,x))\,ds \]
satisfies the following two estimates:
\begin{align}
\label{eq:well}
\|Nu\|_{X(t)}
    & \leq C\,\|(u_0,u_1)\|_{\mathcal{A}_{1,1}} + C\|u\|_{X(t)}^p,  \\
\label{eq:contraction}
\|Nu-Nv\|_{X(t)}
    & \leq C\|u-v\|_{X(t)} \bigl(\|u\|_{X(t)}^{p-1}+\|v\|_{X(t)}^{p-1}\bigr)
\end{align}
uniformly with respect to~$t\in [0,\infty)$. Arguing as we did at the end of the proof of Theorem~\ref{Thm:main} from~\eqref{eq:well} it follows that~$N$ maps~$X(t)$ into itself for
small data. These estimates lead to the existence of a unique solution of $u=Nu$. In fact, taking the recurrence  %recurrent?
sequence $u_{-1} = 0,\  u_{j} = N(u_{j-1})$ for  $j=0,1,2,\dots $,
we apply~\eqref{eq:well} with
$\|(u_0,u_1)\|_{\mathcal{A}_{1,1}}=\epsilon$ and we see
inductively that
\begin{equation}\label{eq:recurrence}
\|u_j\|_{X(t)}\le C_1\epsilon,
\end{equation}
where $C_1=2C$ for any $\epsilon \in[0,\epsilon_0]$ with~$\epsilon_0= \epsilon_0(C_1)$ sufficiently small.
\\
Once the uniform estimate~\eqref{eq:recurrence} is checked we
use~\eqref{eq:contraction} once more and find
\begin{equation}\label{6.4.27}
\|u_{j+1}-u_{j}\|_{X(t)}\le C\epsilon^{p-1}
\|u_j-u_{j-1}\|_{X(t)} \leq 2^{-1}\|u_j-u_{j-1}\|_{X(t)}
\end{equation}
for $\epsilon\le\epsilon_0$ sufficiently small. From~\eqref{6.4.27} we get inductively $\|u_j-u_{j-1}\|_{X(t)}\le {C}{2^{-j}}$ so that $\{u_j \}$ is a Cauchy sequence in the Banach space $X(t)$ converging to the unique solution of $N(u)=u$. Since all of the constants are independent of $t$ we can take $t\to \infty$ and we gain the global existence result. Finally, we see that the definition of $\|u\|_{X(t)}$ leads to the decay estimates \eqref{eq:decayulow}-\eqref{eq:decayuxlow}-\eqref{eq:decayutlow}.
\\
Therefore, to complete the proof  it remains only to establish
\eqref{eq:well} and~\eqref{eq:contraction}. More precisely, we put
\begin{equation}\label{eq:M0}
\|v\|_{X_0(t)} \doteq \sup_{0\leq \tau\leq t} \left[
(1+B(\tau,0))^{n/4}
\|v(\tau,\cdot)\|_{L^2}+(1+B(\tau,0))^{n/4+1/2}\|\nabla
v(\tau,\cdot)\|_{L^2} \right],
\end{equation}
and we prove two slightly stronger inequalities than~\eqref{eq:well}
and~\eqref{eq:contraction}, namely,
\begin{align}
\label{eq:well0}
\|Nu\|_{X(t)}
    &\leq C\,\|(u_0,u_1)\|_{\mathcal{A}_{1,1}} + C\|u\|_{X_0(t)}^p,\\
\label{eq:contraction0}
\|Nu-Nv\|_{X(t)}
    & \leq C\|u-v\|_{X_0(t)} \bigl(\|u\|_{X_0(t)}^{p-1}+\|v\|_{X_0(t)}^{p-1}\bigr).
\end{align}
These conditions will follow from the next proposition in which the restriction on the power $p$ and on the dimension $n$ will appear.
\begin{Prop}\label{Prop:Nestimate}
Let us assume \eqref{eq:admplow}. Let~$(u_0,u_1)\in
\mathcal{A}_{1,1}$ and~$u\in X(t)$. For~$j+l=0,1$ it holds:
\begin{align}
\label{eq:NPwell}
&(1+t)^l(1+B(t,0))^{(n/4+j/2)} \|\nabla^j\partial_t^l Nu(t,\cdot)\|_{L^2}
\leq C\,\|(u_0,u_1)\|_{\mathcal{A}_{1,1}} + C \|u\|_{X_0(t)}^p, \\
\nonumber
&(1+t)^l(1+B(t,0))^{(n/4+j/2)} \|\nabla^j\partial_t^l \bigl(Nu(t,\cdot)-Nv(t,\cdot)\bigr) \|_{L^2} \\
\label{eq:NPcontr}
    &\hspace{7cm} \leq C \|u-v\|_{X_0(t)} \bigl(\|u\|_{X_0(t)}^{p-1}+\|v\|_{X_0(t)}^{p-1}\bigr).
\end{align}
\end{Prop}
%
%\subsection{Proof of Proposition~\ref{Prop:Nestimate}}
%
\begin{proof}
We first prove~\eqref{eq:NPwell}. As in the proof of
Theorem~\ref{Thm:main}  we use two different strategies for~$s\in
[0,t/2]$ and~$s\in [t/2,t]$ to control the integral term in~$Nu$. In
particular, we use Matsumura's type estimate
\eqref{eq:MWDu}-\eqref{eq:MWDux}-\eqref{eq:MWDut} for~$m=1$ if~$s\in
[0,t/2]$ and for~$m=2$ (i.e. energy estimates) if~$s\in [t/2,t]$.
Together with \eqref{eq:MWu}-\eqref{eq:MWux}-\eqref{eq:MWut} and Remark~\ref{Rem:bminus1B} we get
\begin{align}
\nonumber
\|\nabla^j\partial_t^l Nu(t,\cdot)\|_{L^2}
    & \leq C
(1+t)^{-l}(1+B(t,0))^{-(n/4+j/2)}\|(u_0,u_1)\|_{\mathcal{A}_{1,1}}\\
\nonumber
    & \quad + C\int_0^{t/2}(b(s))^{-1}(b(t))^{-l}(1+B(t,s))^{-(n/4+j/2+l)}\|f(u(s,\cdot))\|_{L^1\cap
L^2}ds \\
\label{eq:Nest}
    & \quad + C\int_{t/2}^t (b(s))^{-1}(b(t))^{-l}(1+B(t,s))^{-j/2-l}\|f(u(s,\cdot))\|_{L^2}ds
\end{align}
for $j+l=0,1$. By~\eqref{eq:disscontr} we can estimate~$|f(u)|\lesssim |u|^p$, so that
\[ \|f(u(s,\cdot))\|_{L^1\cap L^2}\lesssim \|u(s,\cdot)\|_{L^p}^p+ \|u(s,\cdot)\|_{L^{2p}}^p, \]
and, analogously,
\[ \|f(u(s,\cdot))\|_{L^2}\lesssim \|u(s,\cdot)\|_{L^{2p}}^p. \]
We apply Gagliardo-Nirenberg inequality (see Remark~\ref{Rem:GN0122}
in Appendix~\ref{sec:GN}):
\begin{align}
\label{eq:GNlowp}
\|u(s,\cdot)\|_{L^p}^p
    & \lesssim \|u(s,\cdot)\|_{L^2}^{p(1-\theta(p))}\,\|\nabla u(s,\cdot)\|_{L^2}^{p\theta(p)},\\
\label{eq:GNlow2p}
\|u(s,\cdot)\|_{L^{2p}}^p
    & \lesssim \|u(s,\cdot)\|_{L^2}^{p(1-\theta(2p))}\,\|\nabla u(s,\cdot)\|_{L^2}^{p\theta(2p)},
\end{align}
where %~$\theta(q)$ is as in~\eqref{eq:thetaq}, that is, %why thetaqapp?
\[ \theta(p)=\frac{n}2\,\frac{p-2}{p}, \qquad \theta(2p)=\frac{n}2\,\frac{p-1}p. \]
We remark that the requisite~$\theta(p)\geq0$ implies that~$p\geq2$,
whereas the requisite~$\theta(2p)\leq1$ implies that~$p\leq
p_\GN(n)$ if~$n\geq3$. The main difference with respect to the proof
of Theorem~\ref{Thm:main} is that to apply Gagliardo-Nirenberg
inequality we need $p\geq2$, since we use the $L^p\cap L^{2p}$ norm
of~$u$ and not its~$L^{p+1}$ norm.
\\
We estimate $\|f(u(s,\cdot))\|_{L^1\cap L^2}$ and $\|f(u(s,\cdot))\|_{L^2}$ by using \eqref{eq:GNlowp}, \eqref{eq:GNlow2p} and $\|u\|_{X_0(t)}$:
\begin{align}
\label{eq:fusGNlowp}
\|f(u(s,\cdot))\|_{L^1\cap L^2}
    & \lesssim \|u\|_{X_0(s)}^p (1+B(s,0))^{-p(n/4+\theta(p)/2)}= \|u\|_{X_0(s)}^p (1+B(s,0))^{-(p-1)n/2},
\intertext{since~$\theta(p)<\theta(2p)$, whereas}
\label{eq:fusGNlow2p}
\|f(u(s,\cdot))\|_{L^2}
    & \lesssim \|u\|_{X_0(s)}^p (1+B(s,0))^{-p(n/4+\theta(2p)/2)}= \|u\|_{X_0(s)}^p (1+B(s,0))^{-(2p-1)n/4}.
\end{align}
Summarizing we find
\begin{align*}
\|\nabla^j\partial_l Nu(t,\cdot)\|_{L^2}
    & \leq C (1+t)^{-l}(1+B(t,0))^{-(n/4+j/2)} \epsilon \\
    & \quad + C\|u\|_{X_0(t)}^p \int_0^{t/2} (b(s))^{-1}(b(t))^{-l} (1+B(t,s))^{-(n/4+j/2+l)}(1+B(s,0))^{-(p-1)n/2}\, ds \\
    & \quad + C\|u\|_{X_0(t)}^p \int_{t/2}^t (b(s))^{-1}(b(t))^{-l} (1+B(t,s))^{-j/2-l}(1+B(s,0))^{-(2p-1)n/4}\, ds
\end{align*}
for~$j+l=0,1$. First, let~$s\in [0,t/2]$. Due to~\eqref{eq:Btssmall}
and~\eqref{eq:Btsbehav2} we can estimate
\[\int_0^{t/2} (b(s))^{-1}(b(t))^{-l}(1+B(t,s))^{-(n/4+j/2+l)}(1+B(s,0))^{-(p-1)n/2} ds \lesssim (1+B(t,0))^{-(n/4+j/2)}(1+t)^{-l}.\]
%
%\begin{eqnarray*}
%&& \int_0^{t/2} (b(s))^{-1}(b(t)^{-l}(1+B(t,s))^{-l} (1+B(t,s))^{-(n/4+j/2)}(1+B(s,0))^{-(p-1)n/2} ds\\
%&& \quad \lesssim (1+t)^{-l} (1+B(t,0))^{-(n/4+j/2)} \,,
%\end{eqnarray*}
%
Indeed, since~$p>p_\Fuj(n)$ after the change of variables~$r=B(s,0)$
we get
\[ \int_0^{t/2} \frac1{b(s)} (1+B(s,0))^{-(p-1)n/2} \, ds = \int_0^{B(t/2,0)} (1+r)^{-(p-1)n/2} \,dr \leq C. \]
Analogously, for~$s\in [t/2,t]$ by using \eqref{eq:Bs0large} we have
\begin{multline*}
\int_{t/2}^t \frac1{b(s)} \frac1{(b(t))^{l}} (1+B(t,s))^{-j/2-l}(1+B(t,0))^{-(2p-1)n/4} ds \\
\leq C (1+B(t,0))^{-(2p-1)n/4} \frac1{(b(t))^l} \int_{t/2}^t
\frac1{b(s)} (1+B(t,s))^{-j/2-l}ds.
\end{multline*}
Thanks to \eqref{eq:intlarge1}-\eqref{eq:intlargej}-\eqref{eq:intlargel} in the proof of Theorem~\ref{Thm:main} we get
\begin{multline*}
\frac1{(b(t))^l} (1+B(t,0))^{-(2p-1)n/4} \int_{t/2}^t \frac1{b(s)} (1+B(t,s))^{-j/2-l}ds \\
\leq C (1+B(t,0))^{-(2p-1)n/4 +1-j/2-l}
(b(t))^{-l}(\log(1+B(t,0)))^l \lesssim
(1+B(t,0))^{-n/4-j/2}(1+t)^{-l}.
\end{multline*}
By using %~\eqref{eq:Btsbehav2} and
Remark~\ref{Rem:bminus1B} we prove~\eqref{eq:NPwell} once we get
\[ (1+B(t,0))^{1-(p-1)n/2} (\log(1+B(t,0)))^l \leq C\,,  \qquad l=0,1\, \]
as follows being~$p>p_\Fuj(n)$.
\\
Now we prove~\eqref{eq:NPcontr}. We remark that
\[ \|Nu-Nv\|_{X(t)} = \left\| \int_0^t E_1(t,s,x)\ast_{(x)} (f(u(s,x))-f(v(s,x))) \,ds \right\|_{X(t)}. \]
Thanks to \eqref{eq:MWDu}-\eqref{eq:MWDux}-\eqref{eq:MWDut} we can estimate
\begin{multline*}
\| \nabla^j \partial_t^l E_1(t,s,x) \ast_{(x)} (f(u(s,x))-f(v(s,x))) \|_{L^2} \\
\lesssim \begin{cases}
(b(s))^{-1}(b(t))^{-l}(1+B(t,s))^{-\frac{j}2-l-\frac{n}4} \|f(u(s,\cdot))-f(v(s,\cdot)) \|_{L^1\cap L^2}, & s\in [0,t/2], \\
(b(s))^{-1}(b(t))^{-l}(1+B(t,s))^{-\frac{j}2-l}
\|f(u(s,\cdot))-f(v(s,\cdot)) \|_{L^2}, & s\in [t/2,t],
\end{cases}
\end{multline*}
for~$j+l=0,1$. By using~\eqref{eq:disscontr} and H\"older's inequality we can now
estimate
\begin{align*}
\|f(u(s,\cdot))-f(v(s,\cdot)) \|_{L^1} & \lesssim \|u(s,\cdot)-v(s,\cdot)\|_{L^p} \, \left(\|u(s,\cdot)\|_{L^p}^{p-1}+ \|v(s,\cdot)\|_{L^p}^{p-1}\right),\\
\|f(u(s,\cdot))-f(v(s,\cdot)) \|_{L^2} & \lesssim
\|u(s,\cdot)-v(s,\cdot)\|_{L^{2p}} \,
\left(\|u(s,\cdot)\|_{L^{2p}}^{p-1}+
\|v(s,\cdot)\|_{L^{2p}}^{p-1}\right).
\end{align*}
Analogously to the proof of~\eqref{eq:well0} we apply
Gagliardo-Nirenberg inequality to the terms
\[ \|u(s,\cdot)-v(s,\cdot)\|_{L^q}, \qquad \|u(s,\cdot)\|_{L^q}\,, \quad \|v(s,\cdot)\|_{L^q}, \]
with~$q=p$ and~$q=2p$, and we conclude the proof of~\eqref{eq:contraction0} by using the assumption $p>p_{\Fuj}(n)$ and the convergence of the integrals in
\eqref{eq:intlarge1}-\eqref{eq:intlargej}-\eqref{eq:intlargel}.
\end{proof}

%%%%%%%%%%%%%%%%%%%%%%%%%%%%%%%%%%%%%%%%%%%%%%
%
%%%%%%%%%%%%%%%%%%%%%%%%%%%%%%%%%%%%%%%%%%%%%%

\section{Proof of Theorem~\ref{Thm:linmain}}

In order to prove Theorem~\ref{Thm:linmain} we follow the strategy in~\cite{W07}. The main goal is to show how the strategy can be extended to a parameter-dependent family of Cauchy problems. For additional details we refer the reader to that paper.

We will prove a statement more general than
\eqref{eq:MWDu}-\eqref{eq:MWDux}-\eqref{eq:MWDut}, namely, that
\begin{eqnarray}\label{eq:Matsugeneral}
&&\|\partial_t^l\partial_x^\alpha v(t,\cdot) \|_{L^2}\\
 &&\quad \le C (b(s))^{-1}(1+B(t,s))^{-\frac{|\alpha|}2-\frac{n}2\left(\frac1m-\frac12\right)}(b(t))^{-l} (1+B(t,s))^{-l}
 \|g(s,\cdot)\|_{L^m\cap H^{[|\alpha|+l-1]^+}},\nonumber
\end{eqnarray}
for~$l=0,1$ and for any~$\alpha\in \N^n$. The inequality~\eqref{eq:Matsugeneral} for~$|\alpha|\leq 1-l$ gives us \eqref{eq:MWDu}-\eqref{eq:MWDux}-\eqref{eq:MWDut}.
\\
We perform the Fourier transform of~\eqref{eq:CPD} and we make the change of variables
\begin{equation}\label{eq:ychange}
y(t,\xi)\doteq \frac{\lambda(t)}{\lambda(s)} \widehat{v}(t,\xi)\,, \qquad \text{where} \quad \lambda(t)\doteq \exp\left(\frac12\int_0^tb(\tau)\,d\tau\right)\,,
\end{equation}
so that we derive the Cauchy problem
\begin{equation}\label{eq:CPv}
y''+m(t,\xi)y=0\, \qquad y(s,\xi)=0, \quad y'(s,\xi)=\widehat{\dat}(s,\xi),
\end{equation}
where we put
\[ m(t,\xi)\doteq \xii^2-\left(\frac14b^2(t)+\frac12b'(t)\right). \]
Let us define~$\bs(t)\doteq b(t)/2$ and
\[ \<\xi\>_{\bs(t)}\doteq \sqrt{\abs{\xii^2-\bs^2(t)}}.\]
We divide the extended phase space $[s,\infty)\times\Rn$ into four zones. We define the following \emph{hyperbolic}, \emph{pseudo-differential}, \emph{reduced} and \emph{elliptic} zones in correspondence of sufficiently small~$\ee>0$ and sufficiently large~$N>0$:
\begin{align*}
Z_\hyp(N)
    & =\Big\{ t\ge s,\; \xii\geq \bs(t), \quad \frac{\<\xi\>_{\bs(t)}}{\bs(t)}\geq N \Big\},\\
Z_\pd(N,\ee)
    & =\Big\{ t\ge s,\; \xii\geq \bs(t), \quad \ee \leq \frac{\<\xi\>_{\bs(t)}}{\bs(t)}\leq N \Big\},\\
Z_\red(\ee)
    & =\Big\{ t\ge s,\; \frac{\<\xi\>_{\bs(t)}}{\bs(t)}\leq \ee \Big\},\\
Z_\elli(\ee)
    & =\Big\{ t\ge s,\; \xii\leq \bs(t), \quad \frac{\<\xi\>_{\bs(t)}}{\bs(t)}\geq \ee \Big\}.
\end{align*}
\begin{Rem}\label{Rem:bs}
Since~$\bs(t)$ is monotone there exists the limit
\[ \bs_\infty \doteq \lim_{t\to\infty} \bs(t) \in [0,\infty].\]
We distinguish the following four cases:
\begin{itemize}
\item If~$\bs(t)\searrow0$, then for any~$\xi\neq0$ there exists~$T_\xii\geq s$ such that~$(t,\xi)\in Z_\hyp(N)$ for any~$t\geq T_\xii$.
\item If~$\bs(t)\searrow \bs_\infty>0$, then for any~$\xii>\bs_\infty\sqrt{N^2+1}$ there exists~$T_\xii\geq s$
such that~$(t,\xi)\in Z_\hyp(N)$ for any~$t\geq T_\xii$. Moreover, $(t,\xi)\in Z_\elli(\ee)$ for any~$\xii\leq
\bs_\infty\sqrt{1-\ee^2}$ and~$(t,\xi)\in Z_\hyp(N)$ for any $\xii\geq\bs(s)\sqrt{N^2+1}$.
\item If~$\bs(t)\nearrow \bs_\infty>0$, then for any~$\xii<\bs_\infty\sqrt{1-\ee^2}$ there exists~$T_\xii\geq s$ such that~$(t,\xi)\in Z_\elli(N)$
for any~$t\geq T_\xii$. Moreover, $(t,\xi)\in Z_\elli(\ee)$ for any~$\xii\leq \bs(s)\sqrt{1-\ee^2}$ and~$(t,\xi)\in Z_\hyp(N)$ for any $\xii\geq \bs_\infty\sqrt{N^2+1}$.
\item If~$\bs(t)\nearrow \infty$, then for any~$\xi\in\Rn$ there exists~$T_\xii\geq s$ such that~$(t,\xi)\in Z_\elli(N)$ for any~$t\geq T_\xii$.
\end{itemize}
\end{Rem}
We define
\[ h(t,\xi)=\chi \bigg(\frac{\<\xi\>_{\bs(t)}}{\ee\bs(t)}\bigg) \ee \bs(t) + \bigg(1-\chi \bigg(\frac{\<\xi\>_{\bs(t)}}{\ee\bs(t)}\bigg)\bigg) \sqrt{\abs{m(t,\xi)}}, \]
where~$\chi\in\mathcal{C}^\infty[0,+\infty)$ localizes: $\chi(\zeta)=1$ if
$0\leq\zeta\leq1/2$ and $\chi(\zeta)=0$ if~$\zeta\geq1$. For any~$(t,\xi)\not\in Z_\red(\ee)$ it holds~$|m(t,\xi)|\geq C\ee^2\bs^2(t)$. Therefore,
$h(t,\xi)\geq C_1\ee\bs(t)$.

Let~$V(t,\xi)=(ih(t,\xi)y(t,\xi),y'(t,\xi))^T$. From~\eqref{eq:CPv} we obtain
\begin{equation}\label{eq:CPV}
V' = \begin{pmatrix}
h'(t,\xi)/h(t,\xi) & ih(t,\xi) \\
im(t,\xi)/h(t,\xi) & 0
\end{pmatrix} V, \qquad V(s,\xi) = (0,\widehat{\dat}(s,\xi))^T.
\end{equation}
For any $t\ge t_1\ge s$ we denote by $\mathcal E(t,t_1,\xi)$ the fundamental solution of~\eqref{eq:CPV}, that is, the matrix which solves
\begin{equation}\label{eq:CPV-E}
\partial_t \mathcal E(t,t_1,\xi) =
\begin{pmatrix}
h'(t,\xi)/h(t,\xi) & ih(t,\xi) \\
im(t,\xi)/h(t,\xi) & 0
\end{pmatrix}  \mathcal E(t,t_1,\xi)\,, \qquad \mathcal E(t_1,t_1,\xi) = I
\end{equation}
for any $t\ge t_1$. It is clear that $V(t,\xi)=\mathcal E(t,s,\xi)(0,\widehat{\dat}(s,\xi))^T$ and that~$\mathcal E(t,t_2,\xi)=\mathcal E(t,t_2,\xi)\,\mathcal E(t_2,t_1,\xi)$, for any~$t\geq t_2\geq t_1\geq s$.
\\
For~$t_2\ge t_1$ and $(t_2,\xi), (t_1,\xi)\in Z_{hyp}(N,\ee)$, we will write $\mathcal E(t_2,t_1,\xi)=\mathcal E_{hyp}(t_2,t_1,\xi)$. Similarly for the other zones.

%%%%%%%%%%%%%%%%%%%%%%%%%%%%

\subsection{Diagonalization in the hyperbolic zone}

Recalling the definition of~$\chi$, in~$Z_\hyp(N)$ it holds $h(t,\xi)=\sqrt{m(t,\xi)}$. Therefore we can write the system in~\eqref{eq:CPV} as
\begin{equation}\label{eq:Vhyp} \partial_t V = i\sqrt{m(t,\xi)}\begin{pmatrix}
0 & 1 \\
1 & 0
\end{pmatrix} V + \frac{\partial_t\sqrt{m(t,\xi)}}{\sqrt{m(t,\xi)}} \begin{pmatrix}1&0\\0&0\end{pmatrix} V.
\end{equation}
The constant matrix
\[ P=\frac1{\sqrt{2}}\begin{pmatrix}
1 & 1 \\
-1 & 1
\end{pmatrix}\,, %, \qquad P^{-1}=\frac1{\sqrt{2}}\begin{pmatrix}
%1 & -1 \\
%1 & 1
%\end{pmatrix} 
\]
is the diagonalizer of the principal part of~\eqref{eq:Vhyp}, that is,
\[ P\begin{pmatrix}0&1\\1&0\end{pmatrix}P^{-1}= \frac1{\sqrt{2}}\begin{pmatrix}1&1\\-1&1\end{pmatrix}P^{-1}=\begin{pmatrix}1&0\\0&-1\end{pmatrix}\,. \]
If we put~$W(t,\xi)=PV(t,\xi)$, then~\eqref{eq:Vhyp} becomes
\begin{equation}\label{eq:Vhyp1} \partial_t W = i\sqrt{m(t,\xi)}\begin{pmatrix}
1 & 0\\
0 & -1
\end{pmatrix} W + \frac{\partial_t\sqrt{m(t,\xi)}}{2\sqrt{m(t,\xi)}} \begin{pmatrix}1&-1\\-1&1\end{pmatrix} W.
\end{equation}
Then we apply a step of refined diagonalization to~\eqref{eq:Vhyp1}.
The second diagonalizer depends on~$\sqrt{m(t,\xi)}$ and
on~$\partial_t\sqrt{m(t,\xi)}$. For this reason there will appear
terms where also~$\partial_t^2\sqrt{m(t,\xi)}$ comes into play. By
using~\eqref{eq:oscb} for~$k=1,2,3$ (we recall that both~$b(t)$ and~$b'(t)$ appear
in the definition of~$m(t,\xi)$) we derive suitable estimates for
the entries of the new system.
\\
Summarizing, for any $(t_1,\xi), (t_2,\xi)\in Z_\hyp(N)$ with $t_1\leq t_2$, the fundamental solution in~\eqref{eq:CPV-E} can be written as
\[ \mathcal{E}_\hyp(t_2,t_1,\xi) = \widetilde{\mathcal{E}}_{\hyp,0}(t_2,t_1,\xi)Q_\hyp(t_2,t_1,\xi), \]
where
\[ \widetilde{\mathcal{E}}_{\hyp,0}(t_2,t_1,\xi)=\diag \left(\exp \bigl(-i\int_{t_1}^{t_2} \sqrt{m(\tau,\xi)}\,d\tau \bigr), \exp \bigl(i\int_{t_1}^{t_2} \sqrt{m(\tau,\xi)}\,d\tau \bigr)\right), \]
and~$\|Q_\hyp(t_2,t_1,\xi)\|\leq C$, uniformly. We remark that in the last estimate we used the property
\[ m(t_2,\xi)\approx \xii \approx m(t_1,\xi)\,, \]
which holds in~$Z_\hyp(N)$, to control the term
\[ \exp \left( \frac12 \int_{t_1}^{t_2} \frac{\partial_\tau\sqrt{m(\tau,\xi)}}{\sqrt{m(\tau,\xi)}} d\tau \right) = \left( \frac{m(t_2,\xi)}{m(t_1,\xi)} \right)^{1/4}, \]
which appears after the refined diagonalization step.

%%%%%%%%%%%%%%%%%%%%%%%%%%%%%%%%%%%%%%%%%

\subsection{Diagonalization in the elliptic zone}

In~$Z_\elli(\ee)$ it holds $h(t,\xi)=\sqrt{-m(t,\xi)}$, therefore we can write the system in~\eqref{eq:CPV} as
\begin{equation}\label{eq:Velli} \partial_t V = i\sqrt{-m(t,\xi)}\begin{pmatrix}
0 & 1 \\
-1 & 0
\end{pmatrix} V + \frac{\partial_t\sqrt{-m(t,\xi)}}{\sqrt{-m(t,\xi)}} \begin{pmatrix}1&0\\0&0\end{pmatrix} V.
\end{equation}
The constant matrix
\[ \widetilde{P}=\frac1{\sqrt{2}}\begin{pmatrix}
-i & 1 \\
i & 1
\end{pmatrix}\,, %\qquad \widetilde{P}^{-1}=\frac1{\sqrt{2}}\begin{pmatrix}
%i & -i \\
%1 & 1
%\end{pmatrix}, 
\]
is the diagonalizer of the principal part of~\eqref{eq:Velli}. If we
put~$W(t,\xi)=\widetilde{P}V(t,\xi)$, then~\eqref{eq:Velli} becomes
\begin{equation}\label{eq:Velli1} \partial_t W = \sqrt{-m(t,\xi)}\begin{pmatrix}
1 & 0\\
0 & -1
\end{pmatrix} W + \frac{\partial_t\sqrt{-m(t,\xi)}}{2\sqrt{-m(t,\xi)}} \begin{pmatrix}1&-1\\-1&1\end{pmatrix} W.
\end{equation}
If~$t_1\geq \overline{t}$ with a sufficiently
large~$\overline{t}\geq s$, then we can perform a step of
\emph{refined} diagonalization. On the other hand, since the subzone
\[ Z_\comp(\ee,\overline{t})= \{t\leq \overline{t}\} \cap Z_\elli(\ee) \subset [s,\overline{t}]\times
\left\{\xii\leq \max\{\bs(s),\bs(\overline{t})\} \right\} \,, \]
is compact, the fundamental solution is bounded there. So we may
assume~$t_1\geq \overline{t}$. For any $(t_1,\xi), (t_2,\xi)\in
Z_\elli(\ee)$ with $t_1\leq t_2$ the fundamental solution
in~\eqref{eq:CPV-E} can be written as
\[ \mathcal{E}_\elli(t_2,t_1,\xi) = \widetilde{\mathcal{E}}_{\elli,0}(t_2,t_1,\xi)Q_\elli(t_2,t_1,\xi), \]
where
\[ \widetilde{\mathcal{E}}_{\elli,0}(t_2,t_1,\xi)=\left(\frac{m(t_2,\xi)}{m(t_1,\xi)}\right)^{1/4}\diag \left(\exp \bigl(\int_{t_1}^{t_2}
\sqrt{-m(\tau,\xi)}\,d\tau \bigr), \exp \bigl(-\int_{t_1}^{t_2}
\sqrt{-m(\tau,\xi)}\,d\tau \bigr)\right) \]
and~$\|Q_\elli(t_2,t_1,\xi)\|\leq C$ uniformly. We remark that the term
\[ \exp \left( \frac12 \int_{t_1}^{t_2} \frac{\partial_\tau\sqrt{m(\tau,\xi)}}{\sqrt{m(\tau,\xi)}} d\tau \right) = \left( \frac{m(t_2,\xi)}{m(t_1,\xi)} \right)^{1/4} \,, \]
which appears in~$\widetilde{\mathcal{E}}_{\elli,0}(t_2,t_1,\xi)$ is
not bounded. Consequently, it can not be included
in~$Q_\elli(t_2,t_1,\xi)$ as we did during the diagonalization
procedure in~$Z_\hyp(N)$.

%%%%%%%%%%%%%%%%%%%%%%%%%%%%%%%%%%%%%%%%%%%%

\subsection{Estimates in the reduced and pseudo-differential zones}

In~$Z_\red(\ee)$ we can estimate~$\sqrt{\abs{m(t,\xi)}}\leq
C\ee\bs(t)$ and therefore also $h(t,\xi)\leq C\ee\bs(t)$.
By rough estimates
this implies %\footnote{In~\cite{W07} there is a misprint, the term $\bs'(t)/\bs(t)$ is missing. In~\cite{W05} it appears in the correct way}
\[ \|\mathcal{E}_\red(t_2,t_1,\xi)\|\leq \exp\left( C\ee\int_{t_1}^{t_2}\bs(\tau)\,d\tau\right).\]
Since~$C$ is independent of~$\ee$ we can take~$\ee<1/(2C)$,
so that the exponential growth is slower than the growth of $\lambda(t_2)/\lambda(t_1)$.
\\
In~$Z_\pd(N,\ee)$ it holds~$h(t,\xi)=\sqrt{m(t,\xi)}$. We can roughly estimate by the symbol class of $\partial_t\sqrt{m(t,\xi)}/\sqrt{m(t,\xi)}$:
\[ \|\mathcal{E}_\pd(t_2,t_1,\xi)\|\leq\exp\left(c\int_{t_1}^{t_2} (1+\tau)^{-1}\,d\tau\right) = \left(\frac{1+t_2}{1+t_1}\right)^c \leq C_\ee' \exp\left( C\ee\int_{t_1}^{t_2}\bs(\tau)\,d\tau\right)\,, \]
for any~$\ee>0$, since~$t\bs(t)\to\infty$.

%%%%%%%%%%%%%%%%%%%%%%%%%%%%%%%%%%%%%%%%%%%%

\subsection{Representation of the solution}

We come back to our original problem~\eqref{eq:CPD}. Let
\[ y(t,s,\xi)=\Psi(t,s,\xi)\widehat{\dat}(s,\xi) \]
be the solution to~\eqref{eq:CPv}. Then, thanks to our
representation for the fundamental
solution~$\mathcal{E}(t_2,t_1,\xi)$ given in~\eqref{eq:CPV-E}, we derive
\[ \begin{pmatrix}
0&  i\xii \Psi \\
0 & \Psi'
\end{pmatrix} (0,\widehat{\dat}(s,\xi))^{T} = \diag (\xii/h(t,\xi), 1) \mathcal{E}(t,s,\xi) \diag (0,1) (0,\widehat{\dat}(s,\xi))^{T}, \]
that is,
\[ \Psi(t,s,\xi) = -i \mathcal{E}_{12}(t,s,\xi) / h(t,\xi), \qquad \Psi'(t,s,\xi) = \mathcal{E}_{22}(t,s,\xi). \]
We write the Fourier transform of the solution to~\eqref{eq:CPD} as $\widehat{v}(t,\xi)=\widehat{\Phi}(t,s,\xi)\widehat{\dat}(s,\xi)$. Recalling~\eqref{eq:ychange}, we obtain
\begin{align}
\label{eq:Phirep}
\widehat{\Phi}(t,s,\xi)
    & = \frac{\lambda(s)}{\lambda(t)} \Psi(t,s,\xi) = -i \frac{\lambda(s)}{\lambda(t)} \frac1{h(t,\xi)} \mathcal{E}_{12}(t,s,\xi), \\
\nonumber
\widehat{\Phi}'(t,s,\xi)
    & = \frac{\lambda(s)}{\lambda(t)} \left(\Psi'(t,s,\xi)-\frac12b(t)\Psi(t,s,\xi) \right) \\
\label{eq:Phi1rep}
    & = \frac{\lambda(s)}{\lambda(t)} \left( \mathcal{E}_{22}(t,s,\xi) + \frac{i b(t)}{2h(t,\xi)} \mathcal{E}_{12}(t,s,\xi)\right).
\end{align}
According to Remark~\ref{Rem:bs}, for any frequency~$\xi\neq 0$ and
initial time~$s\geq0$ (with no loss of generality we can
assume~$s\geq \overline{t}$) we can distinguish various cases. We
first consider the case of~$\bs(t)$ decreasing,
$\bs(t)\searrow\bs_\infty$ with~$\bs_\infty\in [0,+\infty)$,
and~$(s,\xi)\in Z_\elli$, that is, $\xii
\leq\eta(s)\sqrt{1-\ee^2}$.
\begin{itemize}
\item If~$|\xi|>\bs_\infty\sqrt{N^2+1}$, then there exist~$t_\pd>t_\red>t_\elli\geq s$ such that for any~$t\geq t_\pd$ it follows that
\[ \mathcal{E}(t,s,\xi)=\mathcal{E}_\hyp(t,t_\pd,\xi)\mathcal{E}_\pd(t_\pd,t_\red,\xi)\mathcal{E}_\red(t_\red,t_\elli,\xi)\mathcal{E}_\elli(t_\elli,s,\xi). \]
In particular, this happens for any frequency~$\xi\neq0$
if~$\bs_\infty=0$.
\item If~$\bs_\infty\sqrt{1+\ee^2}<|\xi|\leq\bs_\infty\sqrt{N^2+1}$, then there exist~$t_\red>t_\elli\geq s$ such that for any~$t\geq t_\red$ it follows that
\[ \mathcal{E}(t,s,\xi)=\mathcal{E}_\pd(t,t_\red,\xi)\mathcal{E}_\red(t_\red,t_\elli,\xi)\mathcal{E}_\elli(t_\elli,s,\xi). \]
\item If~$\bs_\infty\sqrt{1-\ee^2}<|\xi|\leq\bs_\infty\sqrt{1+\ee^2}$, then there exists ~$t_\elli\geq s$ such that for any~$t\geq t_\elli$ it follows that
\[ \mathcal{E}(t,s,\xi)=\mathcal{E}_\red(t,t_\elli,\xi)\mathcal{E}_\elli(t_\elli,s,\xi). \]
\item If~$\xii\leq\bs_\infty\sqrt{1-\ee^2}$, then~$\mathcal{E}(t,s,\xi)=\mathcal{E}_\elli(t,s,\xi)$.
\end{itemize}
On the other hand, if~$\xii\geq\eta(s)\sqrt{N^2+1}$,
then~$\mathcal{E}(t,s,\xi)=\mathcal{E}_\hyp(t,s,\xi)$ for any~$t\in
[s,\infty)$. The intermediate cases are clear.
\\
If we consider the case of~$\bs(t)$ increasing,
$\bs(t)\nearrow\bs_\infty$ with~$\bs_\infty\in (0,+\infty]$, then
the situation is reversed. In particular, for any frequency~$\xii\in
[\bs(s)\sqrt{N^2+1},\bs_\infty\sqrt{1-\ee^2})$ (if this set is
not empty), there exist~$t_\red>t_\pd>t_\hyp\geq s$ such that for
any~$t\geq t_\red$ it follows that
\[ \mathcal{E}(t,s,\xi)=\mathcal{E}_\elli(t,t_\red,\xi)\mathcal{E}_\red(t_\red,t_\pd,\xi)\mathcal{E}_\pd(t_\pd,t_\hyp,\xi)\mathcal{E}_\hyp(t_\hyp,s,\xi). \]

%%%%%%%%%%%%%%%%%%%%%%%%%%%%%%%%%%%%%%

\subsection{Estimates for the multipliers}

We have to derive estimates for~$|\widehat{\Phi}(t,s,\xi)|$ in each
zone of the extended phase space. The estimates
for~$|\widehat{\Phi}'(t,s,\xi)|$ will be obtained by a more refined
approach. Since~$\mathcal{E}_{12}(t,s,\xi)$ is multiplied by
\[ \frac{\lambda(s)}{\lambda(t)} \frac1{h(t,\xi)} \]
we look in each zone for an estimate of the scalar and non-negative
term
\[ a(t_2,t_1,\xi) \doteq \frac{\lambda(t_1)}{\lambda(t_2)} \frac{h(t_1,\xi)}{h(t_2,\xi)} \|\mathcal{E}(t_2,t_1,\xi)\| \]
for any~$(t_1,\xi),(t_2,\xi)$ in that zone with $t_1 \leq t_2$.
Indeed, from~\eqref{eq:Phirep} it follows
\[ |\widehat{\Phi}(t,s,\xi)| \lesssim \frac1{h(s,\xi)} \, a(t,s,\xi). \]
Following the ideas from the proof to Theorem 17 in~\cite{W07} we
can easily check that the desired estimate in~$Z_\elli(\ee)$ is
\begin{equation}\label{eq:aelli}
a_\elli(t_2,t_1,\xi) \lesssim \exp \left( -C\abs{\xi}^2
\int_{t_1}^{t_2} \frac1{b(\tau)} \,d\tau \right) = \exp
(-C\abs{\xi}^2B(t_2,t_1)).
\end{equation}
We remark that the estimate
\[ \frac{h(t_1,\xi)}{h(t_2,\xi)} \ \frac{(m(t_2,\xi))^{\frac14}}{(m(t_1,\xi))^{\frac14}} \approx \frac{(m(t_1,\xi))^{\frac14}}{(m(t_2,\xi))^{\frac14}}  \]
plays a fundamental role.
\\
In~$Z_\red(\ee)$ it holds $h(t,\xi)\approx \bs(t)\approx |\xi|$
while ~$h(t,\xi)\approx |\xi|$ in~$Z_\pd(N,\ee)$ and
in~$Z_\hyp(N)$. Therefore,  we can
assume~$h(t_1,\xi)/h(t_2,\xi)\approx 1$ in all these zones. The best
estimate is obtained in~$Z_\hyp(N)$.
Since~$\mathcal{E}(t_2,t_1,\xi)$ is bounded we conclude
\begin{align}
\label{eq:ahyp}
a_\hyp(t_2,t_1,\xi)
    & \lesssim \frac{\lambda(t_1)}{\lambda(t_2)}\,;
\intertext{on the other hand, in~$Z_\pd(N,\ee)$ we have}
\label{eq:apd}
a_\pd(t_2,t_1,\xi)
    & \leq \frac{(1+t_2)^{c}}{(1+t_1)^{c}}\frac{\lambda(t_1)}{\lambda(t_2)},
\intertext{whereas in~$Z_\red(\ee)$ we have}
\label{eq:ared}
a_\red(t_2,t_1,\xi)
    & \leq \exp \left(C\ee \int_{t_1}^{t_2} b(\tau)\,d\tau \right) \frac{\lambda(t_1)}{\lambda(t_2)} \equiv \left(\frac{\lambda(t_1)}{\lambda(t_2)}
    \right)^{1-2\delta},
\end{align}
where we choose~$\ee>0$ such that~$\delta\doteq C\ee<1/2$.
It is clear that in the zones~$Z_\hyp(N)$, $Z_\pd(N,\ee)$
and~$Z_\red(\ee)$ we can uniformly estimate~$a(t_2,t_1,\xi)$ by
the upper bound from ~\eqref{eq:ared}, which is the worst among
\eqref{eq:ahyp}-\eqref{eq:apd}-\eqref{eq:ared}. Moreover, we remark
that the parameter~$\xii$ does not come into play in these
estimates. Nevertheless, we should be careful when we compare with
the estimate~\eqref{eq:aelli}, which has a completely different
structure. Having this in mind we define
\[ \Pi_\hyp(\ee) = Z_\red(\ee)\cup Z_\pd(N,\ee) \cup Z_\hyp(N), \]
and we denote by~$t_\xii$ the separating curve
among~$Z_\elli(\ee)$ and~$\Pi_\hyp(\ee)$, that is, the
separating curve between~$Z_\elli(\ee)$
and~$Z_\red(\ee)$). This curve is given by
\[ \bs^2(t_\xii) - \xii^2 = \ee^2 \bs^2(t_\xii), \qquad \text{i.e.} \quad t_\xii = \bs^{-1}\left(\frac{\xii}{\sqrt{1-\ee^2}}\right). \]
We distinguish two cases.
\begin{itemize}
\item For small frequencies~$\xii\leq \bs(s)\sqrt{1-\ee^2}$, since~$h(s,\xi)\approx\bs(s)\approx b(s)$ it holds
\begin{align}
\label{eq:Phiestell}
|\widehat{\Phi}(t,s,\xi)|
    & \lesssim \frac1{b(s)} \exp (-C\xii^2B(t,s)) \quad \text{for~$t\leq t_\xii$,} \\
\label{eq:Phiestellhyp}
|\widehat{\Phi}(t,s,\xi)|
    & \lesssim \frac1{b(s)} \exp \left(-C\xii^2B(t_\xii,s)\right) \left(\frac{\lambda(t_\xii)}{\lambda(t)}\right)^{1-2\delta} \quad \text{for~$t\geq t_\xii$.}
\end{align}
We recall that~$t_\xii=\infty$ if~$\xii\leq\bs_\infty\sqrt{1-\ee^2}$ (in particular, this is trivially true if~$\bs(t)$ is increasing).
\item For large frequencies~$\xii\geq\bs(s)\sqrt{1-\ee^2}$, since~$h(s,\xi)\approx\xii$ it holds
\begin{align}
\label{eq:Phiesthyp}
|\widehat{\Phi}(t,s,\xi)|
    & \lesssim \frac1{\xii} \left(\frac{\lambda(s)}{\lambda(t)}\right)^{1-2\delta} \quad \text{for~$t\leq t_\xii$,} \\
\label{eq:Phiesthypell}
|\widehat{\Phi}(t,s,\xi)|
    & \lesssim \frac1{\xii} \left(\frac{\lambda(s)}{\lambda(t_\xii)}\right)^{1-2\delta} \exp\left(-C\xii^2B(t,t_\xii)\right) \quad \text{for~$t\geq t_\xii$.}
\end{align}
We recall that~$t_\xii=\infty$ if~$\xii\geq\bs_\infty\sqrt{1-\ee^2}$ (in particular, this is trivially true if~$\bs(t)$ is decreasing).
\end{itemize}

%%%%%%%%%%%%%%%%%%%%%%%%%%%%%%%%%%%%%%

\subsection{Estimates for the time derivative of the multipliers}

We consider~$\widehat{\Phi}'(t,s,\xi)$. In~$\Pi_\hyp(\ee)$ we
directly use the representation ~\eqref{eq:Phi1rep} together
with~$b(t)\lesssim h(t,\xi)$ and~$h(s,\xi)\approx\xii\approx
h(t,\xi)$. Therefore, for large frequencies~$\xii\geq
\bs(s)\sqrt{1-\ee^2}$ and for~$t\leq t_\xii$ we can estimate
\begin{align}
\label{eq:Phi1esthyp}
|\widehat{\Phi}'(t,s,\xi)|
    & \lesssim \left(\frac{\lambda(s)}{\lambda(t)}\right)^{1-2\delta},
\intertext{whereas for small frequencies~$\xii\leq \bs(s)\sqrt{1-\ee^2}$ and~$t\geq t_\xii$ we get}
\label{eq:Phi1estellhyppartial}
|\widehat{\Phi}'(t,s,\xi)|
    & \lesssim |\widehat{\Phi}'(t_\xii,s,\xi)| \left(\frac{\lambda(t_\xii)}{\lambda(t)}\right)^{1-2\delta}.
\end{align}
It remains to estimate two objects:
\begin{itemize}
\item $\widehat{\Phi}'(t,s,\xi)$ in the case of small frequencies~$\xii\leq \bs(s)\sqrt{1-\ee^2}$ for any~$t\leq t_\xii$,
\item $\widehat{\Phi}'(t,s,\xi)$ for large frequencies~$\xii\geq\bs(s)\sqrt{1-\ee^2}$ and for any~$t\geq t_\xii$ (we remark that this
case comes into play only if~$\bs(t)$ is decreasing).
\end{itemize}
A direct estimate for~$\widehat{\Phi}'(t,s,\xi)$ is not appropriate for small frequencies~$\xii\leq \bs(s)\sqrt{1-\ee^2}$ and~$t\leq t_\xii$. Taking account of
\[ \widehat{\Phi}''+|\xi|^2\widehat{\Phi}+b(t)\widehat{\Phi}'=0, \quad \widehat{\Phi}(s,s,\xi)=0,\quad \widehat{\Phi}'(s,s,\xi)=1 \]
and setting~$y(t,\xi)=\widehat{\Phi}'(t,s,\xi)$ we get
\begin{equation}\label{eq:ODEPhi}
y'+b(t)y=|\xi|^2\widehat{\Phi}(t,s,\xi), \quad y(s,\xi)=1.
\end{equation}
This leads to the integral equation
\[ y(t,\xi)=\exp\left(-\int_s^tb(\tau)\,d\tau\right) \left(y(s,\xi)+\int_s^t \exp\left(\int_s^\tau b(\sigma)\,
d\sigma\right) |\xi|^2\widehat{\Phi}(\tau,s,\xi)\,d\tau\right), \]
that is,
\[ \widehat{\Phi}'(t,s,\xi)=\frac{\lambda^2(s)}{\lambda^2(t)} +\int_s^t \frac{\lambda^2(\tau)}{\lambda^2(t)} |\xi|^2\widehat{\Phi}(\tau,s,\xi)\,d\tau. \]
Analogously to Lemma 20 in~\cite{W07} we can prove that
\begin{equation}\label{eq:Phi1est}
|\widehat{\Phi}'(t,s,\xi)| \lesssim \frac{\abs{\xi}^2}{b(s)b(t)}
\exp \big(-C|\xi|^2 B(t,s)\big).
\end{equation}
Indeed, by using~\eqref{eq:Phiestell} in the integral and applying
integration by parts (we remark that
$b(\tau)\lambda^2(\tau)/\lambda^2(t)=\partial_\tau
(\lambda^2(\tau)/\lambda^2(t))$) we get
\begin{align*}
|\widehat{\Phi}'(t,s,\xi)|
    & \lesssim \frac{\lambda^2(s)}{\lambda^2(t)} +\int_s^t \left(\frac{\lambda^2(\tau)}{\lambda^2(t)}b(\tau)\right)
    \left( \frac{|\xi|^2}{b(s)b(\tau)} \exp\big(-C|\xi|^2B(\tau,s)\big) \right) \,d\tau \\
    & = \frac{\lambda^2(s)}{\lambda^2(t)} + \frac{|\xi|^2}{b(s)b(t)} \exp\big(-C|\xi|^2B(t,s)\big) -\frac1{b(s)}
    \int_s^t \frac{\lambda^2(\tau)}{\lambda^2(t)} \partial_\tau \left(\frac{|\xi|^2}{b(\tau)} \exp\big(-C|\xi|^2B(\tau,s)\big)\right)d\tau.
\end{align*}
One can show that for $\bs(t)$ increasing or decreasing the second
term determines the desired estimate. Therefore we
derive~\eqref{eq:Phi1est}. Combined
with~\eqref{eq:Phi1estellhyppartial} this allow us to derive for
small frequencies~$\xii\leq \bs(s)\sqrt{1-\ee^2}$  the following
estimates:
\begin{align}
\label{eq:Phi1estell}
|\widehat{\Phi}'(t,s,\xi)|
    & \lesssim \frac{\abs{\xi}^2}{b(s)b(t)} \exp (-C|\xi|^2 B(t,s)) \quad \text{for~$t\leq t_\xii$,} \\
\label{eq:Phi1estellhyp}
|\widehat{\Phi}'(t,s,\xi)|
    & \lesssim \frac{\xii}{b(s)} \exp (-C|\xi|^2 B(t_\xii,s)) \left(\frac{\lambda(t_\xii)}{\lambda(t)}\right)^{1-2\delta} \quad \text{for~$t\geq t_\xii$.}
\end{align}
We remark that we used~$b(t_\xii)\approx \xii$ in~\eqref{eq:Phi1estellhyp}.

\bigskip

To estimate~$\widehat{\Phi}'(t,s,\xi)$ for large
frequencies~$\xii\geq\bs(s)\sqrt{1-\ee^2}$ and for any~$t\geq
t_\xii$ we slightly modify this approach. Indeed, we still
put~$y(t,\xi)=\widehat{\Phi}'(t,s,\xi)$, but now we look for an
estimate of the solution to
\begin{equation}\label{eq:ODEPhitxi}
\begin{cases}
y'+b(t)y=|\xi|^2\widehat{\Phi}(t,s,\xi), \quad t\geq t_\xii, \\
y(t_\xii,\xi)=\widehat{\Phi}'(t_\xii,s,\xi).
\end{cases}
\end{equation}
By using~\eqref{eq:Phiesthypell} for~$\widehat{\Phi}(t,s,\xi)$
and~\eqref{eq:Phi1esthyp} for~$\widehat{\Phi}'(t_\xii,s,\xi)$ we
derive for~$t\geq t_\xii$ the following inequality:
\begin{align*}
|\widehat{\Phi}'(t,s,\xi)|
    & \lesssim \frac{\lambda^2(t_\xii)}{\lambda^2(t)}
    \left[ \left(\frac{\lambda(s)}{\lambda(t_\xii)}\right)^{1-2\delta} +
    \int_{t_\xii}^t \frac{\lambda^2(\tau)}{\lambda^2(t_\xii)} \, \xii^2 \left(\frac1\xii \left(\frac{\lambda(s)}{\lambda(t_\xii)}\right)^{1-2\delta} \,
    \exp\big(-C|\xi|^2B(\tau,t_\xii)\big)\right) \,d\tau \right] \\
    & \lesssim \left(\frac{\lambda(s)}{\lambda(t_\xii)}\right)^{1-2\delta} \left[ \frac{\lambda^2(t_\xii)}{\lambda^2(t)}
    +\frac1\xii \int_{t_\xii}^t \left(\frac{\lambda^2(\tau)}{\lambda^2(t)}b(\tau)\right)\left( \frac{|\xi|^2}{b(\tau)}
    \exp\big(-C|\xi|^2B(\tau,t_\xii)\big) \right) \,d\tau \right].
\end{align*}
We can now easily follow the previous reasoning. Therefore, we
derive for large frequencies~$\xii\geq\bs(s)\sqrt{1-\ee^2}$ and
for any~$t\geq t_\xii$ the estimate
\begin{equation}\label{eq:Phi1esthypell}
|\widehat{\Phi}'(t,s,\xi)| \lesssim
\left(\frac{\lambda(s)}{\lambda(t_\xii)}\right)^{1-2\delta}
\frac{\xii}{b(t)} \exp\big(-C|\xi|^2 B(t,t_\xii)\big).
\end{equation}

%%%%%%%%%%%%%%%%%%%%%%%%%%%%%%%%%%%%%%%%%%%%

\subsection{Small frequencies and large frequencies}

We are now in position to prove the following statement.
\begin{Lem}\label{Lem:freq}
For any~$s \in [0,\infty)$ and for any~$t\geq s$ let us define
\[ \Theta(t,s)\doteq \max\{\bs(s),\bs(t)\} \sqrt{1-\ee^2}. \]
Then the estimates \eqref{eq:Phiesthyp}-\eqref{eq:Phi1esthyp} hold
for any~$\xii\geq\Theta(t,s)$, whereas for
any~$\xii\leq\Theta(t,s)$, we have the following:
\begin{align}
\label{eq:Phiellall}
|\widehat{\Phi}(t,s,\xi)|
    & \lesssim \frac1{b(s)} \exp\big(-C'\xii^2B(t,s)\big),\\
\label{eq:Phi1ellall}
|\widehat{\Phi}'(t,s,\xi)|
    & \lesssim \frac{\xii^2}{b(s)b(t)} \exp\big(-C'\xii^2B(t,s)\big).
\end{align}
\end{Lem}
\begin{Rem}
The small frequencies~$\xii\leq\Theta(t,s)$ are the ones such
that~$(s,\xi)\in Z_\elli(\ee)$ or~$(t,\xi)\in
Z_\elli(\ee)$, whereas the large
frequencies~$\xii\geq\Theta(t,s)$ are the ones for which
both~$(s,\xi),(t,\xi)\in \Pi_\hyp(\ee)$.
\end{Rem}
\begin{proof}
The first part of Lemma~\ref{Lem:freq} is trivial
since~$\xii\geq\Theta(t,s)$ means that $(s,\xi)\in Z_\hyp(\ee)$
and~$t\leq t_\xii$.
\\
To prove \eqref{eq:Phiellall}-\eqref{eq:Phi1ellall}
for~$\xii\leq\theta(t,s)$ we distinguish three cases:
\begin{enumerate}[(A)]
\item\label{en:ellell} $\xii\leq \max \{\bs(s),\bs(t)\} \sqrt{1-\ee^2}$;
\item\label{en:ellhyp} $\bs$ is decreasing and $\bs(t)\sqrt{1-\ee^2}\leq\xii\leq \bs(s)\sqrt{1-\ee^2}$;
\item\label{en:hypell} $\bs$ is increasing and $\bs(s)\sqrt{1-\ee^2}\leq\xii\leq\bs(t)\sqrt{1-\ee^2}$.
\end{enumerate}
In the case~\eqref{en:ellell} the two conditions
\eqref{eq:Phiellall}, \eqref{eq:Phi1ellall} coincide with
\eqref{eq:Phiestell}, \eqref{eq:Phi1estell}.
\\
Now let ~$\bs(t)$ be a decreasing function. Since ~$b(\tau)\lesssim
\xii \lesssim b(\sigma)$ for any~$\tau\leq t_\xii\leq \sigma$ it
holds
\begin{align*}
\exp \bigl( -C_1\xii^2 B(t_\xii,s) \bigr) + \left(\frac{\lambda(t_\xii)}{\lambda(t)}\right)^{2C_2}
    & = \exp \left(-C_1\xii^2 \int_s^{t_\xii} \frac1{b(\tau)}d\tau -C_2 \int_{t_\xii}^t b(\sigma)\,d\sigma \right) \,d\xi \\
    & \leq \exp \left( -\min\{C_1,C_2\} \xii^2 B(t,s) \right).
\intertext{So \eqref{eq:Phiellall}, \eqref{eq:Phi1ellall}
immediately follows from \eqref{eq:Phiestellhyp},
\eqref{eq:Phi1estellhyp} in the case~\eqref{en:ellhyp}. Let
~$\bs(t)$ be an increasing function. Since ~$b(\sigma)\lesssim \xii
\lesssim b(\tau)$ for any~$\tau\leq t_\xii\leq \sigma$ it holds}
\left(\frac{\lambda(s)}{\lambda(t_\xii)}\right)^{2C_1} - \exp \bigl( C_2\xii^2 B(t_\xii,s) \bigr)
    & = \exp \left(-C_1\xii^2 \int_s^{t_\xii} b(\tau)d\tau - C_2 \int_{t_\xii}^t \frac1{b(\sigma)} \,d\sigma \right) \,d\xi \\
    & \leq \exp \left( -\min\{C_1,C_2\} \xii^2 B(t,s) \right).
\end{align*}
Then \eqref{eq:Phiellall}, \eqref{eq:Phi1ellall} follows from
\eqref{eq:Phiesthypell}, \eqref{eq:Phi1esthypell} by using
~$b(s)\lesssim\xii$ in the case~\eqref{en:hypell}.
\end{proof}

%%%%%%%%%%%%%%%%%%%%%%%%%%%%%%%%%%%%%%%%%%%

\subsection{Matsumura-type estimates}

In order to estimate the~$L^2$ norm
of~$\partial_t^l\partial_x^\alpha \Phi(t,s,x)\ast_{(x)}g(s,x)$
for~$l=0,1$ and for any~$|\alpha|\geq0$ we follow the ideas
in~\cite{Matsu} and we distinguish between small and large
frequencies. We fix~$t\in [s,\infty)$.
\begin{Lem}\label{Lem:Matsularge}
The following estimate holds for large
frequencies~$\xii\geq\Theta=\Theta(t,s)$$:$
\begin{equation}\label{eq:Matsularge}
\|
\xii^{|\alpha|}\partial_t^l\widehat{\Phi}(t,s,\cdot)\widehat{g}(s,\cdot)
\|_{L^2_{\{|\xi|\geq \Theta\}}} \lesssim \frac1{b(s)}
\left(\frac{\lambda(s)}{\lambda(t)}\right)^{1-2\delta} \| g(s,\cdot)
\|_{H^{[|\alpha|+l-1]^+}}
\end{equation}
for~$l=0,1$ and for any~$|\alpha|\geq0$, where~$[x]^+$ denotes the
positive part of~$x$.
\end{Lem}
\begin{proof}
First, let~$|\alpha|+l\geq1$. We can estimate
\[ \| \xii^{|\alpha|}\partial_t^l\widehat{\Phi}(t,s,\cdot)\widehat{g}(s,\cdot) \|_{L^2_{\{|\xi|\geq \Theta\}}}
\leq \| |\xi|^{1-l}
\partial_t^l\widehat{\Phi}(t,s,\cdot)\|_{L^\infty_{\{|\xi|\geq
\Theta\}}} \| |\xi|^{|\alpha|+l-1} \widehat{g}(s,\cdot)
\|_{L^2_{\{|\xi|\geq \Theta\}}} \]
for any~$|\alpha|+l\geq1$ since~$|\xi|\leq\<\xi\>$. The second term
can be estimated by~$\|g(s,\cdot)\|_{H^{|\alpha|+l-1}}$. Thanks to
the estimates \eqref{eq:Phiesthyp}, \eqref{eq:Phi1esthyp}, namely
\[ |\partial_t^l\widehat{\Phi}(t,s,\xi)|\lesssim \xii^{-1+l} (\lambda(s)/\lambda(t))^{1-2\delta}, \]
we get a decay uniformly in~$\xii\geq\Theta$ which is given by
\[ \left(\frac{\lambda(s)}{\lambda(t)}\right)^{1-2\delta} = \exp \left(-(1/2-\delta)\int_s^t b(\tau)\,d\tau \right). \]
Now let~$|\alpha|=l=0$. If~$\bs_\infty>0$, then~$\Theta(t,s)\geq
C=\bs_\infty\sqrt{1-\ee^2}>0$ for any~$s,t$, and we can follow
the reasoning above since~$|\xi|^{-1}\approx \<\xi\>^{-1}$ uniformly
in~$\xii\geq C$. Otherwise, if~$\bs(t)\to0$, then after recalling
that~$b(s)\lesssim\xii$ for large frequencies we can estimate
\[ \| \widehat{\Phi}(t,s,\cdot)\widehat{g}(s,\cdot) \|_{L^2_{|\xi|\geq \Theta}} \lesssim \frac1{b(s)}
\left(\frac{\lambda(s)}{\lambda(t)}\right)^{1-2\delta} \|g(s,\cdot)
\|_{L^2}. \]
This completes the proof. \end{proof}
\begin{Rem}
If~$\bs(t)\to\bs_\infty>0$, or if we are interested into an estimate
for~$s\in [0,S]$ and~$t\geq s$ for some fixed~$S>0$, then
~$\Theta(t,s)$ is uniformly bounded by a positive constant.
Therefore (see the proof of Lemma~\ref{Lem:Matsularge}), we can
replace $\|g(s,\cdot) \|_{H^{[|\alpha|+l-1]^+}}$ in the
estimate~\eqref{eq:Matsularge} by~ $\|g(s,\cdot)
\|_{H^{|\alpha|+l-1}}$, that is, by $\|g(s,\cdot)\|_{H^{-1}}$ in the
case~$|\alpha|=l=0$.
\\
In particular, this is possible if we are only interested in estimates for~$s=0$. This explains the difference in the regularity of the initial data~$(0,g(s,\cdot))$ if we compare \eqref{eq:MWu} ($L^m\cap H^{-1}$ regularity) and~\eqref{eq:MWDu} ($L^m\cap L^2$ regularity).
\end{Rem}
\begin{Lem}
The following estimate holds for small frequencies
~$\xii\leq\Theta=\Theta(t,s)$$:$
\begin{equation}\label{eq:Matsusmall}
\|
|\xi|^{|\alpha|}\partial_t^l\widehat{\Phi}(t,s,\cdot)\widehat{g}(s,\cdot)
\|_{L^2_{\{|\xi|\leq\Theta\}}}\lesssim \frac1{b(s)}
(B(t,s)b(t))^{-l}
(B(t,s))^{-\frac{|\alpha|}2-\frac{n}2\left(\frac1m-\frac12\right)}
\|g(s,\cdot)\|_{L^m}
\end{equation}
for~$l=0,1$ and for any~$|\alpha|\geq0$.
\end{Lem}
\begin{proof}
Let~$m'$ and~$p$ be defined by~$1/m+1/m'=1$ and~$1/p+1/m'=1/2$, that is, $1/p=1/m-1/2$. We can estimate
\[ \| |\xi|^{|\alpha|}\partial_t^l\widehat{\Phi}(t,s,\cdot)\widehat{g}(s,\cdot) \|_{L^2_{\{|\xi|\leq\Theta\}}}
\leq \| |\xi|^{|\alpha|}
\partial_t^l\widehat{\Phi}(t,s,\cdot)\|_{L^p_{\{|\xi|\leq\Theta\}}} \|
\widehat{g}(s,\cdot) \|_{L^{m'}_{\{|\xi|\leq\Theta\}}}. \]
We can control~$\| \widehat{g}(s,\xi) \|_{L^{m'}}$
by~$\|g(s,\cdot)\|_{L^m}$. So we have to control the~$L^p$ norm of
the multiplier. Thanks to \eqref{eq:Phiellall},
\eqref{eq:Phi1ellall} we can estimate
\[ \| |\xi|^{|\alpha|} \partial_t^l\widehat{\Phi}(t,s,\cdot)\|_{L^p_{\{|\xi|\leq\Theta\}}}
\lesssim \frac1{b(s)(b(t))^l} \left(\int_{\{\xii\leq\Theta\}}
\xii^{p(|\alpha|+2l)} \exp\big(-Cp\xii^2 B(t,s)\big) \,d\xi
\right)^{\frac1p}. \]
Let~$\rho=Cp|\xi|^2B(t,s)$. After a change of variables to spherical
harmonics (the term~$\rho^{n-1}$ appears) we conclude
\[\int_{\{\xii\leq\Theta\}} \xii^{p(|\alpha|+2l)} \exp\big(-Cp\xii^2 B(t,s)\big) \,d\xi \lesssim (B(t,s))^{-(p(|\alpha|+2l)+n)/2}
\int_0^\infty \rho^{p(|\alpha|+2l)+n-1} e^{-\rho} \,d\rho.\]
We remark that the case~$\Theta(t,s)\to\infty$ brings no additional difficulties. The integral is bounded and we get a decay given by
\begin{equation}\label{eq:Matsudecay}
\frac1{b(s)(b(t))^l} (B(t,s))^{-|\alpha|/2-l-n/(2p)} = \frac1{b(s)}
(B(t,s)b(t))^{-l}
(B(t,s))^{-\frac{|\alpha|}2-\frac{n}2\left(\frac1m-\frac12\right)}.
\end{equation}
The proof is finished. \end{proof}
One can easily check that the decay function given in~\eqref{eq:Matsusmall} is worst than the one in~\eqref{eq:Matsularge}. Therefore, gluing together~\eqref{eq:Matsularge} and~\eqref{eq:Matsusmall}, we derive~\eqref{eq:Matsugeneral}. This concludes the proof of Theorem~\ref{Thm:linmain}.

%%%%%%%%%%%%%%%%%%%%%%%%%%%%%%%%%%%%%%%%%%
%
%%%%%%%%%%%%%%%%%%%%%%%%%%%%%%%%%%%%%%%%%%

\section{Generalizations and improvements}\label{sec:Gen}

\subsection{Admissible damping terms}
We may include oscillations in the damping term~$b(t)u_t$ if we replace Hypotheses~\ref{Hyp:b} and~\ref{Hyp:furtherb} by the following.
\begin{Hyp}\label{Hyp:bbs}
We assume that ~$b=b(t)$ satisfies the  conditions \eqref{en:bpos}-\eqref{en:bsnotdiss}-\eqref{en:breg}-\eqref{en:1b} in Hypothesis~\ref{Hyp:b}. Moreover,
we assume the existence of an \emph{admissible shape function} ~$\bs: [0,\infty)\to [0,\infty)$ such that
\[ \abs{\frac{b(t)}{\bs(t)}-2} \lesssim \frac1{1+t},\]
and  $\bs\in\mathcal{C}^1$, $\bs(t)>0$, monotone, and $t\bs(t)\to\infty$ as~$t\to\infty$. Finally, it satisfies ~\eqref{eq:tbprimeb}, that is,
$t\bs'(t)\leq m\bs(t)$ for some~$m\in [0,1)$.
\end{Hyp}
Then the statements of Theorem~\ref{Thm:linmain} and Theorems~\ref{Thm:main} and~\ref{Thm:low} are still valid.
\begin{Rem}
Let us assume that we have a life-span estimate for the local solution to~\eqref{eq:diss}, which guarantees that~$T_m(\epsilon)\to\infty$ as~$\epsilon\to0$, where~$T_m=T_m(\epsilon) \in (0,\infty]$ is the maximal existence time (see Lemma~\ref{Lem:localweight}). Then condition~\eqref{eq:tbprimeb} in Hypothesis~\ref{Hyp:furtherb} can be weakened to
\begin{equation}
\label{eq:limsupcrit} l\doteq\limsup_{t\to\infty} \frac{t\bs'(t)}{\bs(t)} <1,
\end{equation}
that is, it holds
\begin{equation}\label{eq:tbprimebt0}
t\bs'(t)/\bs(t)\leq m <1, \quad t\geq t_0
\end{equation}
for some~$t_0\geq0$, where we take~$m\in (l,1)$. Indeed, there exists~$\epsilon_1(t_0)>0$ such that~$T_m(\epsilon)\geq
2t_0$ for any~$\epsilon\in(0,\epsilon_1(t_0)]$, and this allow us to rewrite the proof of Theorems~\ref{Thm:main} and~\ref{Thm:low} starting from $t_0$.
\end{Rem}

%%%%%%%%%%%%%%%%%%%%%%%%%%%%%%%%%%

\subsection{Semi-linear damped wave equation with small data in~$L^m\cap H^1$}
An intermediate case between the~$L^2$ framework in~\cite{NO} and the~$L^1$ context in~\cite{IMN} has been studied in~\cite{IO}.
For initial data in~$\mathcal{A}_{m,1}$, the authors find the critical exponent~$p(n,m)=1+(2m)/n$ for~$n\leq 6$,
for any~$m\in (1,2)$ if~$n=1,2$ and for suitable~$m\in[\overline{m},\overline{\overline{m}})$ if~$3\leq n\leq 6$.
\\
If we consider data $(u_0,u_1)\in \mathcal{A}_{m,1}$, for some~$m\in(1,2)$, then we can follow~\cite{IO} to extend Theorem~\ref{Thm:low}.
The range of admissible exponents for the nonlinear term will also depend on the choice of~$m\in(1,2)$.

%%%%%%%%%%%%%%%%%%%%%%%%%%%%%%%%%%%%%
%

\appendix

\section{Gagliardo - Nirenberg inequality}\label{sec:GN}

Here we state some Gagliardo-Nirenberg type inequalities which come into play in the proofs of Theorems~\ref{Thm:main} and~\ref{Thm:low}.
\begin{Lem}[Gagliardo-Nirenberg inequality, see Theorem 9.3 in~\cite{Fried}, Part 1]\label{Lem:GN}
Let~$j,m\in\N$ with~$j<m$, and let~$u\in\mathcal{C}^{m}_c(\Rn)$,
i.e.~$u\in\mathcal{C}^m$ with compact support. Let~$a\in [j/m,1]$, and
let~$p,q,r$ in~$[1,\infty]$ be such that
\[ j-\frac{n}q = \left(m-\frac{n}r\right)a-\frac{n}p(1-a). \]
Then
\begin{equation}\label{eq:GN}
\|D^ju\|_{L^q} \leq C_{n,m,j,p,r,a} \|D^mu\|_{L^r}^a\ \|u\|_{L^p}^{1-a}
\end{equation}
provided that
\begin{equation}\label{eq:GNcond}
\left(m-\frac{n}r\right)-j \not\in\N,
\end{equation}
i.e. $n/r>m-j$ or $n/r\not\in\N$. If~\eqref{eq:GNcond} is not satisfied, then~\eqref{eq:GN} holds provided that~$a\in [j/m,1)$.
\end{Lem}
\begin{Rem}\label{Rem:GN0122}
If~$j=0$, $m=1$ and~$r=p=2$, then~\eqref{eq:GN} reduces to
\begin{equation}\label{eq:GN0122}
\|u\|_{L^q} \lesssim \|\nabla u\|_{L^2}^{\theta(q)}\ \|u\|_{L^2}^{1-\theta(q)},
\end{equation}
where~$\theta(q)$ is given from
\begin{equation}\label{eq:thetaqapp}
-\frac{n}q = \left(1-\frac{n}2\right)\theta(q)-\frac{n}2(1-\theta(q)) = \theta(q) -\frac{n}2,
\end{equation}
that is, $\theta(q)$ is as in~\eqref{eq:thetaq}. It is clear that~$\theta(q)\geq0$ if and only if $q\geq2$. Analogously ~$\theta(q)\leq1$ if and only if
\begin{equation}\label{eq:qgn}
\text{either $n=1,2$\,\,\, or} \quad q\leq 2^*\doteq \frac{2n}{n-2}.
\end{equation}
Applying a density argument the inequality \eqref{eq:GN0122} holds for any $u\in H^1$. Assuming $q<\infty$ the condition \eqref{eq:GNcond} can be neglected
also for $n=2$. Summarizing the estimate~\eqref{eq:GN0122} holds for any finite $q\geq2$ if~$n=1,2$ and for any $q\in [2,2^*]$ if~$n\geq3$.
\end{Rem}
In weighted spaces ~$H^1_{\psi(t,\cdot)}$ we can derive the following statements:
\begin{Lem}\label{Lem:GNweight}
Let $q\ge 2$ be such that \eqref{eq:qgn} holds, and let~$\theta(q)$ be as in~\eqref{eq:thetaqapp}. We have the following properties for any $\sigma \in
[0,1]$ and $t\ge 0$:
\begin{enumerate}[(i)]
\item\label{it:1} Let $\psi \ge 0$. If $v\in H^1_\psi$,  then $v\in H^1_{\sigma \psi}$ and for $j=0,1$ one has
\[
\|e^{\sigma \psi(t,\cdot)}\nabla^j v(t,\cdot)\|_2\le \|\nabla^j v\|_2^{1-\sigma}\|e^{\psi(t,\cdot)}\nabla^j v(t,\cdot)\|_2^\sigma.
\]
\item\label{it:2} Let $\Delta \psi \ge 0$. If $v\in H^1_{\sigma \psi}$, then $e^{\sigma\psi(t,\cdot )}v\in H^1$ and
\[
\|\nabla (e^{\sigma\psi(t,\cdot )}v)\|_2\le \|e^{\sigma\psi(t,\cdot )}\nabla v\|_2.
\]
\item\label{it:3} Let $\Delta \psi \ge 0$. If $v\in H^1_\psi$,  then
\[
\|e^{\sigma \psi(t,\cdot)}v\|_{L^q}\lesssim \|e^{\sigma \psi(t,\cdot)}v\|_{L^2}^{1-\theta(q)}\|e^{\sigma \psi(t,\cdot)}\nabla v\|_{L^2}^{\theta(q)}.
\]
\item\label{it:4} Let $\psi \ge 0$ such that $\inf_{x \in \R^n} \Delta\psi(t,x)=:C(t)>0$. Then
\[
\|e^{\sigma \psi(t,\cdot)}v\|_{L^q}\le (C(t))^{-\frac{1-\theta(q)}2}\|e^{\sigma\psi(t,\cdot )}\nabla v\|_2.
\]
\end{enumerate}
\end{Lem}
\begin{proof}
The statement~\eqref{it:1} is trivial for $\sigma=0$ and requires only H\"older's inequality for $\sigma \in (0,1]$. The property~\eqref{it:2} is obtained
by integration by parts, see Lemma 2.3 in \cite{IT}. For~\eqref{it:3} one combines~\eqref{it:2} with a Gagliardo-Nirenberg inequality (Lemma~\ref{Lem:GN}).
For~\eqref{it:4} one combines~\eqref{it:3} with integration by parts used in proving~\eqref{it:2}.
\end{proof}

%%%%%%%%%%%%%%%%%%%%%%%%%%%%%%%%%%%%%%%%%%%%%%

%%%%%%%%%%%%%%%%%%%%%%%%%%%%%%%%%%%%%

\section*{Acknowledgments}

The first and the third author have been supported by a grant of DFG
(Deutsche Forschungsgemeinschaft) for the research project
\emph{Influence of time-dependent coefficients on semi-linear wave
models} (RE 961/17-1).

%%%%%%%%%%%%%%%%%%%%%%%%%%%%%%%%%%%%%
%
%%%%%%%%%%%%%%%%%%%%%%%%%%%%%%%%%%%%%

\end{document}